\documentclass[journal,manuscript,onecolumn,12pt]{IEEEtran}

\usepackage{setspace}
\usepackage{amsmath,amsfonts}
\usepackage{algorithmic}
\usepackage{algorithm}
\usepackage{array}
\usepackage[caption=false,font=normalsize,labelfont=sf,textfont=sf]{subfig}
\usepackage{textcomp}
\usepackage{stfloats}
\usepackage{url}
\usepackage{verbatim}
\usepackage{graphicx}
\usepackage{cite}
\RequirePackage{amsthm,amssymb,mathrsfs}
\RequirePackage[numbers]{natbib}
\RequirePackage[colorlinks,citecolor=blue,urlcolor=blue]{hyperref}
\usepackage{enumitem}  
\usepackage{euscript}       
\usepackage{color}
\usepackage{bm,bbm}
\usepackage{soul}

\newtheorem{theorem}{Theorem}
\newtheorem{lemma}{Lemma}
\newtheorem{definition}{Definition}

\theoremstyle{remark}
\newtheorem{remark}{Remark}
 
\DeclareMathOperator*{\essup}{ess\,sup}
\newcommand{\Real}{{\mathbb{R}}}
\newcommand{\Exp}{\mathsf{E}}
\newcommand{\Pro}{\mathsf{P}}

\newcommand{\cF}{\mathscr{F}}

\newcommand{\1}{\mathbbm{1}}
\newcommand{\ind}[1]{\1_{\{#1\}}}
\newcommand{\half}{\frac{1}{2}}

\newcommand{\I}{\mathsf{I}}
\newcommand{\J}{\mathsf{J}}
\newcommand{\F}{\mathsf{F}}
\newcommand{\Q}{\mathsf{Q}}
\newcommand{\f}{\mathsf{f}}

\newcommand{\Tc}{\mathsf{T}}
\newcommand{\hth}{\hat{\theta}}

\newcommand{\T}{\EuScript{T}}
\newcommand{\U}{\EuScript{U}}
\newcommand{\E}{\EuScript{E}}
\newcommand{\R}{\EuScript{R}}
\newcommand{\N}{\EuScript{N}}
\renewcommand{\L}{\EuScript{L}}
\renewcommand{\S}{\EuScript{S}}
\renewcommand{\P}{\mathsf{P}}

\begin{document}

\title{Window-Limited CUSUM for Sequential \\ Change Detection}

\author{Liyan~Xie, 
George V.~Moustakides,
and Yao~Xie
\thanks{The work of Yao Xie is partially supported by an NSF CAREER CCF-1650913, and NSF DMS-2134037, CMMI-2015787, CMMI-2112533, DMS-1938106, and DMS-1830210. {\it (Corresponding
author: Liyan Xie.)}

Liyan Xie is with School of Data Science, The Chinese University of Hong Kong, Shenzhen 518172, China (e-mail: xieliyan@cuhk.edu.cn). 

George V. Moustakides is with the Electrical and Computer Engineering Department, University of Patras, 26500 Patras, Greece (e-mail: moustaki@upatras.gr). 

Yao Xie is with the H. Milton Stewart School of Industrial and Systems Engineering, Georgia Institute of Technology, Atlanta,
GA 30332, USA (e-mail: yao.xie@isye.gatech.edu).}
}



\maketitle

\begin{abstract}
We study the parametric online changepoint detection problem, where the underlying distribution of the streaming data changes from a known distribution to an alternative that is of a known parametric form but with unknown parameters. We propose a joint detection/estimation scheme, which we call Window-Limited CUSUM, that combines the cumulative sum (CUSUM) test with a sliding window-based consistent estimate of the post-change parameters. We characterize the optimal choice of the window size and show that the Window-Limited CUSUM enjoys first-order asymptotic optimality as average run length approaches infinity under the optimal choice of window length. Compared to existing schemes with similar asymptotic optimality properties, our test can be much faster computed because it can recursively update the CUSUM statistic by employing the estimate of the post-change parameters. A parallel variant is also proposed that facilitates the practical implementation of the test. Numerical simulations corroborate our theoretical findings.
\end{abstract}




%


\section{Introduction}\label{sec:intro}

Online changepoint detection is a fundamental problem in statistics and signal processing \cite{poor-hadj-QCD-book-2008,Siegmund1985,tartakovsky2014sequential,tutorial_jsait} which finds applications in a plethora of practical problems in diverse fields. The most common version of the problem consists of a sequence of observations sampled independently. There is also a changepoint such that the underlying distribution changes from one distribution to an alternative. This problem is of major importance in many applications, such as seismic signal processing \cite{xie2019asynchronous}, industrial quality control \cite{shi2009quality}, dynamical systems monitoring \cite{lai1995sequential}, structural health control \cite{balageas2010structural}, event detection \cite{li2017detecting}, anomaly detection \cite{chandola2009anomaly}, detection of attacks \cite{tartakovsky2014rapid}, etc. The goal of online changepoint detection is to detect the occurrence of the change in statistical behavior with a minimal delay while controlling the false alarm rate. The suitable tradeoff between detection delay and false alarm rate, as in all detection problems, is of essential importance for the proper mathematical formulation of the problem.

Classical formulations assume complete knowledge of the pre- and post-change underlying distributions with the cumulative sum (CUSUM) test being the most popular means for the corresponding detection \cite{page-biometrica-1954}. The CUSUM scheme properly updates the log-likelihood ratio between the post- and pre-change densities of the available data to form the corresponding test statistic. CUSUM is known to be theoretically optimum in the sense that it enjoys minimum detection delay under a fixed false alarm rate constraint \cite{mous-astat-1986}. Also, it is computationally simple because an updating formula exists for the computation of the CUSUM statistics, which requires only the current data sample (and the previous value of the CUSUM statistic).

In many real world applications, the post-change distribution is typically not precisely known since it represents a switching to an anomalous state. In this case, a Generalized Likelihood Ratio (GLR) test version of CUSUM has been developed, which applies the GLR method to select the unknown post-change distributions \cite{lai-ieeetit-1998} and form the corresponding test statistic. Unfortunately, this original GLR version turns out to be computationally demanding because it requires computations per sample, which increase linearly in time without limit. The main reason for this disadvantage is that the test statistic must be recomputed using GLR for each possible changepoint location with every new observation. A remedy also proposed in \cite{lai-ieeetit-1998} is the window-limited GLR test where the recomputation of the test statistic is limited to changepoint locations within a window of fixed length starting from the current time instant. The resulting scheme has been shown to enjoy asymptotic optimality with proper selection of the window size. Even though computations are now limited since they are of the order of the selected window length, they still tend to be considerable because for each time instant we need to recompute the GLR statistic for \textit{each} position within the window.

In this article, we develop an alternative approach for solving the problem of interest which we call Window-Limited CUSUM (\hbox{WLCUSUM}). It consists in adopting a window-based estimate of the unknown post-change parameters and, unlike the existing window-limited GLR, we use the estimate in the updating formula of the classical CUSUM statistic. This updating mechanism is far more efficient than its window-limited GLR counterpart and by proper selection of the window size we can also guarantee asymptotic optimality. In detail, we only require the sample size $w\rightarrow\infty$ as $\gamma\rightarrow\infty$ with $w=o(\log\gamma)$ while the window-limited GLR require $w=\Theta(\log\gamma)$ where $\gamma$ is the average run length requirement (see Remark \ref{rem:computation} for details). We would like to emphasize that the problem we consider is not joint detection and estimation as in \cite{chen2013optimal,moustakides2012joint}, where the two tasks are regarded as equally important. Here, we are primarily interested in detection, with estimation being an auxiliary action that contributes towards our detection goal. For this reason we only require the estimator to be consistent without insisting on any explicit form.

Compared with existing CUSUM-like procedures employing estimates of post-change parameters, the proposed WLCUSUM method applies to a far wider range of parametric distributions and not only the exponential family which is mostly the case with the available approaches. Our main contributions in this work include: (i)~Proof of asymptotic optimality of the proposed WLCUSUM procedure under Lorden's {\it worst-case} detection delay \cite{Lorden1971}. To achieve this goal, we had to develop new upper bounds for overshoots over a constant threshold for sums of data that are $w$-dependent. (ii)~Characterization of the {\it optimal} choice of the window size to guide practical implementations and offer the best possible performance for the proposed scheme. (iii)~Development of an alternative parallel version of WLCUSUM capable of matching the performance of the optimal window without the need to explicitly specify it. 

We must also mention that one of the main characteristics of WLCUSUM is its computational efficiency. The benchmark window-limited GLR \cite{lai-ieeetit-1998} requires a window size that is at least as large as the detection delay, while in our detector the optimal window has a size that is significantly smaller. This difference in window size translates into an overall computational complexity which, in our scheme is at least an order of magnitude smaller than the corresponding complexity of the window-limited GLR.

 
The remainder of this paper is organized as follows. Section\,\ref{sec:related} reviews related work. Section\,\ref{sec:setup} introduces the adopted formulation and presents details of the proposed detection procedure. Section\,\ref{sec:theory} contains the theoretical analysis establishing the asymptotic optimality of the proposed procedure and the form of the optimal window size. Section\,\ref{sec:variants} presents a parallel implementation of the proposed WLCUSUM with varying window sizes which is particularly suited for practical implementation. Finally, Section\,\ref{sec:numerical} presents examples that demonstrate the performance of the WLCUSUM procedure with comparisons to the corresponding GLR scheme. For better readability of the main results, most of our technical part is moved to the Appendix.

\section{Related Work}\label{sec:related}

The study of online changepoint detection can be traced back to the early work of Page \cite{page-biometrica-1954} and has been studied for several decades. Most articles consider the problem under independent observations but there are also extensions to more complicated data models; see \cite{detectAbruptChange93,poor-hadj-QCD-book-2008,Siegmund1985,tartakovsky2014sequential,tartakovsky2019sequential,veeravalli2013quickest} for thorough reviews in this field. 

We distinguish two main tracks in sequential change detection: The first is the Bayesian approach, where a prior for the time of change is assumed to be available. The second is the minimax (non-Bayesian) formulation, where the changepoint is considered to be deterministic but unknown. Interestingly, both approaches can be put under the same mathematical framework \cite{Moustakides2008,tartakovsky2010state} and, depending on the data model, we can decide which formulation is most suitable to be adopted. 

The first exact optimality result in sequential change detection can be found in \cite{Shiryaev1963} where the focus is on detecting a change in the drift of a Brownian motion. Following a Bayesian approach, the change time is modeled as an independent random variable that is exponentially distributed. For non-Bayesian approaches, the CUSUM test is perhaps the most popular change detection algorithm for the classical setup of detecting a change from a known nominal to a known alternative density. CUSUM, also known as the Page test \cite{page-biometrica-1954}, was first shown to be asymptotically optimum in \cite{Lorden1971} when observations are i.i.d.~before and after the change. The exact optimality of the CUSUM test under the same data model was established in \cite{mous-astat-1986}. An alternative detection procedure introduced in \cite{poll-astat-1985}, even though it was developed by adopting a minimax approach, presents very strong similarities to the optimum Bayesian test developed in \cite{Shiryaev1963}. This sequential detector, known as the Shiryaev-Roberts-Pollak (SRP) test, enjoys a very strong asymptotic optimality property which, unfortunately, was proven not to be exact \cite{polunchenko2010optimality}. We must also mention that in the classical version of the problem, which we discussed so far, the computation of the corresponding test statistic of the CUSUM and the SRP test is straightforward and can be implemented very efficiently.

When we consider parametric density families where the parameters are unknown, CUSUM and SRP are used as prototypes to develop variants which, at best, can enjoy asymptotic optimality.
In particular, when the pre-change density is known while the post-change contains unknown parameters, there are two classes of tests in the literature that address this problem: (i)~The generalized likelihood ratio (GLR) approach \cite{lai-ieeetit-1998}, where the detection statistic, at each time, is computed by substituting the unknown parameter with its maximum likelihood estimate using all potential post-change data until the current time; (ii)~The mixture likelihood ratio procedure \cite[Pages 418-423]{tartakovsky2014sequential}, where the detection statistic is a weighted average of the corresponding log-likelihood ratio by assuming a weight (prior distribution) on the post-change parameters. Although the corresponding tests, as mentioned, enjoy asymptotic optimality properties, a major drawback is that their computational complexity can be high because their detection statistic cannot be computed efficiently. In addition, the score statistics were also used to avoid estimation of the unknown parameters \cite{kirch2015use}.

For the popular CUSUM test, many variants of the traditional version were proposed to improve the computational efficiency when the post-change density contains unknown parameters. To detect the change over a wide range of mean shifts in quality control, the combined usage of CUSUM and Shewhart charts was employed in \cite{lucas1982combined,westgard1977combined}, and the simultaneous use of multiple CUSUM procedures with different drift values was suggested in \cite{sparks2000cusum,zhao2005dual,yu2020note,romano2021fast}. Moreover, the case with finitely many post-change distributions was considered and the joint detection/isolation algorithm was proposed based on multiple hypotheses sequential probability ratio test \cite{nikiforov1995generalized,lai2000sequential}. A different method known as Adaptive CUSUM was first proposed in \cite{sparks2000cusum} and studied further in \cite{abbasi2019optimal,abbasi2020new,jiang2008adaptive,luo2009adaptive,shu2006markov,wu2009enhanced}. The Adaptive CUSUM continuously adjusts its statistic in order to efficiently signal a one-step-ahead forecast in deviation from its target value. For example, such a procedure has been considered in \cite{cao2018entropy}, where the estimate is based on online algorithms such as stochastic gradient descent, and the performance metrics are related to the regret bounds of the online estimators. The simple exponentially weighted moving average (EWMA) estimate is the most common selection for the one-step-ahead forecast. Optimality properties of the Adaptive CUSUM were considered in \cite{lorden2008sequential} where the first-order asymptotic optimality was established for the univariate exponential family while extensions appear in \cite{wu2017detecting}. Finally, a multi-stream Adaptive CUSUM test was proposed in \cite{xu2021multi} establishing asymptotic optimality for the case of Gaussian distributions.

\section{Preliminaries}\label{sec:setup}

\subsection{Problem Setup}


Suppose we have access to the multivariate data sequence $\{\xi_t\}$ with $\xi_t \in \Real^k$, which is sampled sequentially. We assume that there are two probability density functions (pdf) $\f_\infty(\cdot),\f_0(\cdot)$ and a deterministic time $\tau\in\{0,1,2,\ldots\}$ such that 
\begin{equation}
\xi_t  \stackrel{\rm i.i.d.}{\sim}\left\{\begin{array}{ll}
\f_\infty(\xi),&t = 1,2,\ldots,\tau,\\[2pt]
\f_0(\xi,\theta),&t = \tau+1,\tau+2,\ldots
\end{array}\right. 
\label{eq:data_model}
\end{equation}
In other words $\tau$ is a changepoint, where the observations are i.i.d.~before and including $\tau$ following the pdf $\f_\infty(\cdot)$, while for times after $\tau$ they are again i.i.d.~following $\f_0(\cdot,\theta)$, which is characterized by an unknown parameter vector $\theta\in\Theta\subseteq\Real^K$ where $\Theta$ is a known subset of $\Real^K$. If there is no constraint on $\theta$ then we simply set $\Theta=\Real^K$. Throughout this paper we assume that the pre-change distribution $f_\infty(\cdot)$ does {\it not} belong to the post-change distributions with parameter set $\Theta$.

We denote with $\Pro_\infty,\Exp_\infty$ the probability measure and the corresponding expectation when all samples follow the pre-change distribution (i.e.,~the change happens at $\infty$), $\Pro_0^\theta,\Exp_0^\theta$ when all data are under the post-change density with parameter $\theta$ (the change happens at 0) and finally with $\Pro_\tau^\theta,\Exp_\tau^\theta$ the measure and expectation induced when the change happens at time $\tau$. We also denote with $\xi_{t_1}^{t_2}$ ($t_2\geq t_1$) the collection of data $\{\xi_{t_1},\ldots,\xi_{t_2}\}$.

We assume that $\f_\infty(\cdot)$ is {\it known} since usually it can be estimated from historical data by density estimation \cite{silverman1986density} methods. We can also assume that $\f_\infty(\cdot)$ is partially known by extending our results to incorporate the estimation error for $\f_\infty(\cdot)$. However, we note that this estimation error can be negligible as long as the volume of historical data is sufficiently large. For the post-change density, we assume that $\f_0(\cdot,\theta)$ has a known form, but the parameter vector $\theta$ is unknown, without any prior distribution that can capture its statistical behavior. 
Our goal is to detect the changepoint $\tau$ as quickly as possible from streaming data when it occurs, under the constraint that the false alarm rate is properly controlled.

A sequential change detection test simply consists of a {\it stopping time} $T$ which denotes the time we stop and declare that a change took place before time $T$. The stopping time is adapted to the filtration $\{\cF_t\}$, $\cF_t=\sigma\{\xi_1,\ldots,\xi_t\}$ with $\cF_0$ denoting the trivial sigma-algebra. This assures that only available data are employed when we decide whether to stop or not at each time $t$.
Our intention is to use the classical CUSUM test for the change detection problem. Since CUSUM requires exact knowledge of the pre- and post-change densities, we will replace the unknown parameter vector $\theta$ with a proper estimate over available data. This estimate will be renewed with every new data sample. Before presenting the details of our scheme, let us first recall the CUSUM test and its corresponding optimality properties. 

\subsection{The CUSUM Test}

Fix the post-change parameter vector $\theta$ and suppose it is \textit{known}. This suggests that we are under the classical formulation, where we would like to detect a change from a known density $\f_\infty(\cdot)$ to an alternative known density $\f_0(\cdot,\theta)$. This problem can be solved optimally with the CUSUM test. The CUSUM statistic $\{\mathsf{S}_t\}$ is defined as 
\begin{equation*}
\mathsf{S}_t = \max_{0\leq k<t }\sum_{j=k+1}^t \log\frac{\f_0(\xi_j,\theta)}{\f_\infty(\xi_j)},
\end{equation*}
which is essentially the (maximum) likelihood ratio statistics as detailed in \cite[Section 8.2.3]{tartakovsky2014sequential}. The above CUSUM statistic satisfies and is usually implemented through the following update
\begin{equation}
\mathsf{S}_t=\mathsf{S}_{t-1}^+ + \log\frac{\f_0(\xi_t,\theta)}{\f_\infty(\xi_t)},
~\mathsf{S}_0=0,
\label{eq:statCUSUM}
\end{equation}
where $x^+=\max\{x,0\}$. 
The update in \eqref{eq:statCUSUM} is applied every time a new sample becomes available. The corresponding CUSUM stopping time that signals the change is then defined as
\begin{equation}
{\Tc}=\inf\{t>0: \mathsf{S}_t\geq\nu\},
\label{eq:stCUSUM}
\end{equation}
where $\nu>0$ is a constant threshold, the choice of which needs to balance the false alarm rate and the detection delay. 
The first time $\mathsf{S}_t$ hits or exceeds the threshold $\nu$, we stop and declare that a change took place before $t$. CUSUM is known to solve \textit{exactly} \cite{mous-astat-1986} the following challenging constrained optimization problem
\begin{equation}
\begin{aligned} &\inf_{T}\sup_{\tau\geq0}\,\essup\Exp_{\tau}^\theta[T-\tau|\cF_\tau,T>\tau],\\&\text{subject~to:}~\Exp_\infty[T]\geq\gamma>1,
\end{aligned}
\label{eq:optCUSUM}
\end{equation}
when the threshold $\nu$ is selected to satisfy the false alarm constraint with equality, namely $\Exp_\infty[{\Tc}]=\gamma$. In other words, among all stopping times that have an average false alarm period (also known as {\it average run length}, ARL) no smaller than $\gamma$, the CUSUM stopping time ${\Tc}$ has the smallest \textit{worst-case average detection delay} (WADD). From \eqref{eq:optCUSUM} we observe that for each possible changepoint $\tau$ we consider the average detection delay conditioned on the \textit{worst possible data} before (and including) time $\tau$; this is a particular type of delay measure proposed by Lorden \cite{Lorden1971}. It is well-known that CUSUM regarding this worst-case performance is an equalizer in the sense that $\essup\Exp_{\tau}^\theta[\Tc-\tau|\cF_\tau,\Tc>\tau]$ is the same for all changepoints $\tau$; therefore, for the computation of the worst-case detection delay scenario, we can simply limit ourselves to $\tau=0$ (i.e., $\Exp_0^\theta[\Tc]$). 

The analysis in this article will be {\it asymptotic} (for large $\gamma$). For this reason, with the following lemma, we provide a convenient asymptotic formula for the CUSUM performance.
\begin{lemma}[Performance of exact CUSUM]\label{lem:1}
For threshold $\nu=\log\gamma$ the CUSUM test satisfies
\begin{equation}
\Exp_\infty[\Tc]\geq \gamma,~~~~\Exp_0^\theta[\Tc]=\frac{\log\gamma}{\I_0}\left(1+\varTheta\left(\frac{1}{\log\gamma}\right)\right),
\label{eq:perfCUSUM}
\end{equation}
where $\I_0=\Exp_0^\theta\left[\log\frac{\f_0(\xi_1,\theta)}{\f_\infty(\xi_1)}\right]$ is the Kullback-Leibler information number (divergence) of the post- and pre-change densities. 
\end{lemma}

\begin{proof}
The performance of CUSUM, when expressed in asymptotic terms, is usually given in the form of $\Exp_0^\theta[\Tc]=\frac{\nu}{\I_0}\big(1+o(1)\big)$ (see \cite{Siegmund1985}). However, here we would like to be more explicit regarding the $o(1)$ term in order to be able to compare the case of known versus the case of estimated $\theta$. The proof of this formula can be found in \cite[Lemma 1]{xu2021optimum}.
\end{proof}

From Lemma\,\ref{lem:1} we conclude that if we know $\theta$ and apply the CUSUM test defined in \eqref{eq:statCUSUM}, \eqref{eq:stCUSUM} with threshold $\nu=\log\gamma$ then the corresponding CUSUM stopping time $\Tc$ enjoys an asymptotic performance captured by \eqref{eq:perfCUSUM}. In fact no other stopping time $T$ that satisfies the same false alarm constraint can have a limiting value (liminf) for the ratio $\Exp_0^\theta[T]/\frac{\log\gamma}{\I_0}$ that is smaller than 1 as $\gamma\to\infty$. This statement describes the optimality of CUSUM in first-order asymptotic terms as $\gamma\to\infty$.

\section{Proposed Method: Window-Limited CUSUM Test}

In a realistic case $\theta$ is unknown and, as we mentioned, we may know instead a set $\Theta$ of possible values for $\theta$. Of course $\Theta=\Real^K$ if there is no restriction on $\theta$. If $\theta$ is not exactly known then the CUSUM test cannot be applied in the form of \eqref{eq:statCUSUM}, \eqref{eq:stCUSUM}. For this reason, as in the literature, we propose to replace the unknown $\theta$ with a \textit{consistent estimate}. 
Specifically we select a window of length $w\geq1$ and define the Window-Limited CUSUM (WLCUSUM) test statistic $\{\S_t\}$ for $t>w$ similarly to \eqref{eq:statCUSUM}:
\begin{equation}
\S_t=\S_{t-1}^+ + \log\frac{\f_0(\xi_t,\hth_{t-1})}{\f_\infty(\xi_t)},
\S_w=0, \ t = w+1, w+2, \ldots,
\label{eq:statACUSUM}
\end{equation}
where $\hth_t\in\Theta$ is an estimate of $\theta$. We are not going to adopt any specific estimator, we only constrain the estimate $\hth_t$ to be based on the data $\{\xi_{t},\ldots,\xi_{t-w+1}\}$ and to be consistent.
An obvious possibility would be the Maximum Likelihood Estimator (MLE)
\begin{equation}
\hth_t=\text{arg}\max_{\theta\in\Theta}\sum_{i=0}^{w-1}\log\f_0(\xi_{t-i},\theta),
\label{eq:mlACUSUM}
\end{equation}
which, as we will see in our analysis, enjoys certain desirable optimality characteristics when employed in our proposed detection scheme.
For the stopping time, similarly to CUSUM, we define
\begin{equation}
\T=\inf\{t>w: \S_t\geq\nu\},
\label{eq:stACUSUM}
\end{equation}
with $\nu>0$ is a constant threshold. 				

It is worth noting that the increment term $\{\log\f_0(\xi_t,\hth_{t-1})/\f_\infty(\xi_t)\}_{t\in\mathbf{Z}}$ in \eqref{eq:statACUSUM} is $w$-dependent (defined below); while the increment terms (log-likelihood ratios) in the exact CUSUM \eqref{eq:statCUSUM} are independent. Due to such $w$-dependency induced by the estimates which employ data from the past, obtaining the formulas for the detection performance is not a straightforward task.
\begin{definition}[$w$-dependence, \cite{janson1984runs}] A discrete-time stochastic process $(Y_n)_{n\in\mathbf{Z}}$ is $w$-dependent if for all $k$, the joint stochastic variables $(Y_n)_{n\leq k}$ are independent of the joint stochastic variables $(Y_n)_{n\geq k+w+1}$. 
\end{definition}

We must point out that several existing detection/estimation methods propose to recalculate the estimate $\hth_t$ at each time $t$ and perform a dual maximization. Such is, for example, the popular GLR approach proposed in \cite{lai-ieeetit-1998}
$$
  \tilde{\S}_t=\max_{0\leq\tau<t} \sup_{\theta\in\Theta}\sum_{s=\tau+1}^t \log\frac{\f_0(\xi_s,\theta)}{\f_\infty(\xi_s)}. 
$$
Unfortunately, the above statistic does not possess any convenient updating formula similar to \eqref{eq:statCUSUM} and requires a number of operations per sample that increases linearly with time.
To remedy this serious computational handicap, a {\it window-limited} version is commonly adopted where the search for the maximum over $\tau$ is performed within a window of fixed length $w$. Specifically, the following maximization replaces the previous one
\begin{equation}\label{glr_lai}
  \tilde{\S}_t=\max_{t-w\leq\tau<t} \sup_{\theta\in\Theta}\sum_{s=\tau+1}^t \log\frac{\f_0(\xi_s,\theta)}{\f_\infty(\xi_s)}. 
\end{equation}
This clearly reduces the complexity to a fixed number of operations per time update but, as we discuss in Remark\,\ref{rem:computation}, it can still be quite demanding. In the following, we refer to \eqref{glr_lai} as the window-limited GLR approach.

Unlike the window-limited GLR, we propose to employ, as in \cite{xie2020seq_ana}, the classical CUSUM update in \eqref{eq:statCUSUM} where we simply replace the unknown parameter $\theta$ with a consistent estimate, thus preserving the computational efficiency of the original CUSUM. 
The reason we expect this idea to be successful is that when the data are under the post-change regime, $\hth_{t-1}$ will be close to the true $\theta$ if $w$ is sufficiently large, and the WLCUSUM statistic will exhibit a positive drift not differing significantly from the exact CUSUM drift. On the other hand, when the data follow the pre-change regime, we will show that the estimate will impose a negative drift on the WLCUSUM statistic forcing the test to perform repeated restarts exactly similarly to the case of the exact CUSUM. These claims will be demonstrated through a rigorous analysis. Additionally, we will obtain asymptotic formulas for the average false alarm period and the worst-case average detection delay, which will allow us to establish the asymptotic optimality of our proposed detection scheme. 
Before starting our mathematical derivations, let us make some remarks and present our assumptions.

\begin{remark}\label{rem:1}
As can be seen from \eqref{eq:statACUSUM}, we do not compute any test statistic during the first $w$ samples since we accumulate these samples in order to obtain the first estimate $\hth_w$. The test statistic is first computed at time $w+1$ where we also employ the first estimate $\hth_w$. Since we necessarily wait for $w$ time instances, it is understood that whatever average delay we compute, it cannot be smaller than $w$. Asymptotically, this fact is not disturbing because we will assure in Section\,\ref{sec:theory} that this initial waiting time $w$ is {\it negligible} compared to the actual detection delay required to detect the change. Of course, not applying the test during the first $w$ time instances is only for {\it analytical} purposes, as this corresponds to the worst-case average detection delay (as we prove in Lemma\,\ref{lem:3}). In a practical implementation, we can start testing earlier, and our estimator at every time $t$ can rely on the existing data without necessarily waiting until $w$ samples become available. We discuss this point in more detail in Section\,\ref{sec:variants} when we introduce a computationally convenient variant of our scheme.
\end{remark}

\begin{remark}\label{rem:2}
The estimate $\hth_t$ that we employ in our test, as mentioned, is obtained by processing the samples $\{\xi_t,\ldots,\xi_{t-w+1}\}$. This assures that $\xi_t$ and $\hth_{t-1}$ are {\it independent} and the same is true between $\hth_t$ and $\cF_{t-w}$ ($w$-dependency). As we are going to see, these facts play an important role in our proofs. 
\end{remark}

\begin{remark}\label{rem:name}
We note that the name ``window-limited CUSUM'' has been used in the literature, e.g., \cite{jacob2008sequential} and \cite{lai-ieeetit-1998}. But they are all very different from our proposed scheme. In detail, Lai's definition of window-limited CUSUM statistics is $S_t= \max_{t-w < k < t} \sum_{i=k+1}^t \log\frac{\f_0(\xi_i,\theta)}{\f_\infty(\xi_i)}$, thus it also assumed full knowledge of the pre- and post-change distributions. Therefore, we must differentiate our algorithm from those in the literature that usually refers to the above statistics and may coincides with window-limited GLR in some senses, and is non-recursive, while our proposed WL-CUSUM enjoys recursive update and the estimator $\hth$ is also differently constructed.
\end{remark}

\subsection{Assumptions and Useful Observations}

Regarding the estimates $\{\hth_t\}$ since it is not our intention to promote any specific estimation method, we will impose general characteristics that are enjoyed by most reasonable estimators (such as MLE). In particular, we make the following key assumptions for the estimator.
\begin{itemize}
\item[A1:] Under the post-change regime $\Pro_0^\theta$, we assume that $\Exp_0^\theta[\hth_t]=\theta$ (unbiased estimator\footnote{\label{footnote1}In fact our analysis can also accommodate asymptotically unbiased estimators provided that the norm square of the bias is $o(1/w)$. It is for simplicity that we limit ourselves to the unbiased case.}).
If we write $\hth_t=\theta+\E_t$, the zero mean estimation error has a covariance matrix of the form $\Exp_0^\theta[\E_t\E_t^\intercal]=\frac{1}{w}\Sigma_0\big(1+o(1)\big)$. Matrix $\Sigma_0$ is of the order of a constant when considered as a function of $w$. 

\item[A2:] Under the pre-change regime $\Pro_\infty$, we assume that $\Exp_\infty[\hth_t]=\theta_\infty$. If we write $\hth_t=\theta_\infty +\E_t$, the zero mean estimation error has a covariance matrix of the form $\Exp_\infty[\E_t\E_t^\intercal]=\frac{1}{w}\Sigma_\infty(1+o(1)\big)$. Matrix $\Sigma_\infty$  is of the order of a constant as a function of $w$.
\end{itemize}
With A1, we require our estimator, when applied to post-change data to provide reliable estimates of the correct parameter vector $\theta$. We expect the quality of our estimate to improve with increasing window size $w$, since the error covariance matrix is inversely proportional to $w$. When applied to pre-change data, the estimator behavior is described by A2. We assume that it provides estimates close to some value $\theta_\infty$, where $\theta_\infty$ is estimator dependent. For example, in the case of the MLE, we have that
$$
\theta_\infty=\mathrm{arg}\max_\theta\Exp_\infty[\log\f_0(\xi_1,\theta)],
$$
which is the limiting form of \eqref{eq:mlACUSUM} after we normalize with the window size $w$ and invoke the Law-of-Large Numbers. We note that the MLE satisfies A1 and A2 since MLE is asymptotically unbiased (satisfying the footnote \ref{footnote1}) under mild conditions \cite{lehmann2006theory}. Let us now continue with our assumptions. The next assumption refers to the Kullback-Leibler (KL) information numbers and the second moment of the log-likelihood ratio.
\begin{itemize}
\item[A3:] Consider the two KL information numbers and the second moment of the log-likelihood ratio under the $\Pro_0^\theta$ measure
\begin{align*}  \I_0&=\Exp_0^\theta\left[\log\frac{\f_0(\xi_1,\theta)}{\f_\infty(\xi_1)}\right],\\
\I_\infty&=-\Exp_\infty\left[\log\frac{\f_0(\xi_1,\theta_\infty)}{\f_\infty(\xi_1)}\right],\\
\J_0&=\Exp_0^\theta\left[\left(\log\frac{\f_0(\xi_1,\theta)}{\f_\infty(\xi_1)}\right)^2\right].  
    \end{align*} 
We assume that all three quantities are strictly {\it positive} and {\it bounded} for every $\theta\in\Theta$ of interest.
\end{itemize}
We see that $\I_\infty$ involves the parameter value $\theta_\infty$, which is estimated when the data are under the pre-change regime. This information number can be strictly positive if, for example, the pre-change density $\f_\infty(\cdot)$ cannot be expressed as the post-change density $\f_0(\cdot,\theta_\infty)$ for some particular value $\theta_\infty\in\Theta$ that belongs to the allowable set of post-change parameter values. Unfortunately, the analysis that follows does not cover the case where $\f_\infty(\cdot)=\f_0(\cdot,\theta_\infty)$ and it requires significant technical modifications to address this particular possibility. On the other hand we can avoid the occurrence of this case by defining a suitable set $\Theta$ that does not contain $\theta_\infty$. 

In our derivations, we will encounter quantities that resemble the KL information numbers and the second moment of the log-likelihood ratio introduced in Assumption\,A3, but with the parameter $\theta$ replaced by an estimate. In particular, we will be interested in the following alternatives:
\begin{align*}
\hat{\I}_0&=\Exp_0^\theta\left[\log\frac{\f_0(\xi_t,\hth_{t-1})}{\f_\infty(\xi_t)}\right],\\
\hat{\I}_\infty&=-\Exp_\infty\left[\log\frac{\f_0(\xi_t,\hth_{t-1})}{\f_\infty(\xi_t)}\right],\\
\hat{\!\J}_0&=\Exp_0^\theta\left[\left(\log\frac{\f_0(\xi_t,\hth_{t-1})}{\f_\infty(\xi_t)}\right)^2\right]. 
\end{align*}
Due to Assumptions\,A1 and A2 we expect $\hat{\I}_0$, $\hat{\I}_\infty$, and $\hat{\!\J}_0$ to be close to their exact counterparts $\I_0,\I_\infty,\J_0$. 
In fact we can specify their relationship more precisely by applying a Taylor expansion on $\log\big(\f_0(\xi_t,\hth_{t-1})/\f_\infty(\xi_t)\big)$ around the mean of $\hth_{t-1}$ and retaining the first three terms. Additionally, we could make suitable assumptions on the smoothness of $\f_0(\xi,\theta)$ to guarantee the effectiveness of these approximations. In order to avoid these common technicalities, we propose to simply assume that such an expansion is valid without more details. This will allow us to focus on the more interesting question of the WLCUSUM asymptotic optimality, which will require a number of novel results due to the $w$-dependency of the approximate log-likelihood ratios $\{\log\big(\f_0(\xi_t,\hth_{t-1})/\f_\infty(\xi_t)\big)\}$. Consequently, our last assumption expresses how $\hat{\I}_0,\hat{\I}_\infty,\hat{\!\J}_0$ are related to $\I_0,\I_\infty,\J_0$.
\begin{itemize}
\item[A4:] (Taylor expansion based approximations.) The quantities $\hat{\I}_0$, $\hat{\I}_\infty$, and $\hat{\!\J}_0$ can be written as
\begin{align}\label{eq:assA4}
\begin{split}
\hat{\I}_0&=\I_0-\frac{1}{2w}\mathrm{trace}\{\Sigma_0\F_0\}\big(1+o(1)\big),\\
\hat{\I}_\infty&=\I_\infty+\frac{1}{2w}\mathrm{trace}\{\Sigma_\infty \F_\infty\}\big(1+o(1)\big),\\
\hat{\!\J}_0&=\J_0+\frac{1}{w}\mathrm{trace}\{\Sigma_0\Q_0\}\big(1+o(1)\big),
\end{split}
\end{align}
where $\F_0,\F_\infty,\Q_0$ are matrices of the order of a constant with respect to the window size $w$. 
\end{itemize}
By applying Taylor expansion on the approximate log-likelihood function, it is possible to identify the exact form of $\F_0,\F_\infty,\Q_0$ in \eqref{eq:assA4} for unbiased estimators satisfying Assumptions\,A1 and A2. We present this in the following lemma.

\begin{lemma}\label{lem:QPQ}
The matrices entering in \eqref{eq:assA4} in Assumption\,A4 have the following form
\begin{align*}
\F_0&=\Exp_0^\theta\left[ \left(\frac{\nabla_{\!\theta}\f_0(\xi_1,\theta)}{\f_0(\xi_1,\theta)}\right)\left(\frac{\nabla_{\!\theta}\f_0(\xi_1,\theta)}{\f_0(\xi_1,\theta)}\right)^\intercal\right],\\
\F_\infty&=\Exp_\infty\bigg[\left(\frac{\nabla_{\!\theta}\f_0(\xi_1,\theta_\infty)}{\f_0(\xi_1,\theta_\infty)}\right)\left(\frac{\nabla_{\!\theta}\f_0(\xi_1,\theta_\infty)}{\f_0(\xi_1,\theta_\infty)}\right)^\intercal -\frac{\nabla_{\!\theta\theta}\f_0(\xi_{1},\theta_\infty)}{\f_0(\xi_{1},\theta_\infty)}
\bigg],\\
\Q_0&=\F_0+\Exp_0^\theta\bigg[\left( \log\frac{f_0(\xi_1,\theta)}{f_\infty(\xi_1)}\right)\bigg\{ \frac{ \nabla_{\theta\theta}  f_0(\xi_1,\theta)}{f_0(\xi_1,\theta)} - 
\left(\frac{\nabla_{\!\theta}\f_0(\xi_1,\theta)}{\f_0(\xi_1,\theta)}\right)\left(\frac{\nabla_{\!\theta}\f_0(\xi_1,\theta)}{\f_0(\xi_1,\theta)}\right)^\intercal\bigg\} \bigg],  
\end{align*}
where $\nabla_{\!\theta\theta}\f_0(\xi,\theta)$ denotes the Hessian of\/ $\f_0(\xi,\theta)$ with respect to $\theta$.
\end{lemma}
\begin{proof}
Demonstrating the validity of these formulas presents no particular difficulty. We apply Taylor expansion on $\log\big(\f_0(\xi_t,\hth_{t-1})/\f_\infty(\xi_t)\big)$ with respect to $\hth_{t-1}$ around its mean $\theta$ and retain the first three terms. Then we make use of the independence between $\xi_t$ and $\hth_{t-1}$ and therefore between $\xi_t$ and $\E_{t-1}$ (the estimation error of $\hth_{t-1}$). The first order term in the expansion has average 0 because $\E_{t-1}$ is zero mean. Consequently, we end up with the expectation of the second order term with respect to $\E_{t-1}$ while, as we assumed, higher order terms are regarded as negligible and captured by the $o(1)$ components in \eqref{eq:assA4}. In order to obtain the desired results we must also note that $\Exp_0^{\theta}[\nabla_{\!\theta}\f_0(\xi_1,\theta)/\f_0(\xi_1,\theta)]=0$, $\Exp_0^{\theta}[\nabla_{\!\theta\theta}\f_0(\xi_1,\theta)/\f_0(\xi_1,\theta)]=0$ and that for a deterministic matrix $\Q$, we can write $\Exp[\E_{t-1}^\intercal\Q\E_{t-1}]=\Exp[\mathrm{trace}\{\Q\E_{t-1}\E_{t-1}^\intercal\}]=\mathrm{trace}\{\Q\Exp[\E_{t-1}\E_{t-1}^\intercal]\}$. Computations are straightforward, thus we omit further details. 
\end{proof}
\begin{remark}[Examples of Gaussian distribution]
    We provide a concrete example of the quantities above under Gaussian distributions. Assume one-dimensional Gaussian mean shift from $\N(0,1)$ to $\N(\theta,1)$ with post-change mean equal to 1, and assume the set $\Theta=\{\theta: \theta\geq 0.5\}$ as the possible post-change mean values. When using the maximum likelihood estimator, we have $\hat\theta_t = \frac{1}{w}\sum_{i=0}^{w-1} \xi_{t-i}$ if $\frac{1}{w}\sum_{i=0}^{w-1} \xi_{t-i}\geq 0.5$ and $\hat\theta_t = 0.5$ otherwise. Under the post-change regime, the asymptotic distribution of $\hat\theta_t$ is $\N(\theta,1/w)$, while under the pre-change regime, $\hat\theta_t$ converges in probability to $0.5$. In this case we have $\theta=1$, $\theta_\infty=0.5$, $\I_0=\theta^2/2$, $\I_\infty = 0.5^2/2$, $\J_0=\theta^4/4+\theta^2$, and Fisher information $\F_0=1$, $\F_\infty=1$, $\Q_0=\F_0-\theta^2/2=1/2$, and $\hat \I_0\approx \I_0 -1/(2w)$, $\hat \I_\infty=1/8+o(1/w)$, $\hat \J_0\approx \J_0+1/(2w)$. It is worthwhile commenting that when the pre-change distribution is non-normal, the parameter set $\Theta$ will equal to $\Real$ and the MLE $\hth_t$ is asymptotically normal under mild conditions for both the pre- and post-change regimes.
\end{remark}

Let us now make some useful observations regarding the quantities we introduced above. From Equation \eqref{eq:assA4} we deduce that the window size $w$ must be larger than $\mathrm{trace}\{\Sigma_0\F_0\}/\I_0$ so that $\hat{\I}_0>0$, a property which is necessary for successful detection. Indeed, only when $\hat{\I}_0>0$, the WLCUSUM statistic will exhibit a positive drift after change, forcing the corresponding statistic to increase and finally exceed the positive threshold. Meanwhile, we require the pre-change drift $-\hat{\I}_\infty$ to be negative; this holds in general for any $w$ and any estimate as we discuss in detail in Section\,\ref{sec:theory}.

We note from its definition that $\F_0$ is the Fisher Information matrix. Since for any unbiased estimator we have validity of the Cramer-Rao Lower Bound, namely $\Sigma_0 \succcurlyeq \big(\F_0\big)^{-1}$, we conclude that
\begin{equation}
\mathrm{trace}\{\Sigma_0\F_0\}=\mathrm{trace}\{\F_0^\half\Sigma_0\F_0^\half\}\geq
\mathrm{trace}\big\{\F_0^\half\F_0^{-1}\F_0^\half\big\}=K;
\label{eq:CRLB}
\end{equation}
where, we recall that $K$ is the dimension of the parameter vector $\theta$. Therefore, the maximum likelihood estimate asymptotically {\it maximizes} the approximate KL information number $\hat{\I}_0$, thus resulting in the smallest detection delay among all consistent estimators.

Consider now the \textit{mismatched} version of the KL information number $\Exp_0^\theta\left[\log(\f_0(\xi_1,\vartheta)/\f_\infty(\xi_1))\right]$ with $\vartheta\neq \theta$, then
\begin{align*}
\Exp_0^\theta\left[\log\frac{\f_0(\xi_1,\vartheta)}{\f_\infty(\xi_1)}\right]&=\Exp_0^\theta\left[\log\frac{\f_0(\xi_1,\theta)}{\f_\infty(\xi_1)}\right]+\Exp_0^\theta\left[\log\frac{\f_0(\xi_1,\vartheta)}{\f_0(\xi_1,\theta)}\right] \leq
\Exp_0^\theta\left[\log\frac{\f_0(\xi_1,\theta)}{\f_\infty(\xi_1)}\right]=\I_0.    
\end{align*}
This observation and the fact that $\I_0$ is constant allows us to deduce that
\begin{equation}
\begin{aligned}
\Exp_0^\theta\left[\log\frac{\f_0(\xi_t,\hth_{t-1})}{\f_\infty(\xi_t)}\Big|\cF_{t-1}\right]\leq\I_0  \ \Rightarrow \  
\hat{\I}_0=\Exp_0^\theta\left[\log\frac{\f_0(\xi_t,\hth_{t-1})}{\f_\infty(\xi_t)}\right]\leq\I_0,
\end{aligned}
\label{eq:inequ}
\end{equation}
which will be used in the derivations that follow.

\section{Theoretical Analysis}\label{sec:theory}

We are now ready to analyze the proposed WLCUSUM test. We begin by considering next the average false alarm.
\begin{lemma}[ARL of WLCUSUM]\label{lem:2}
The WLCUSUM defined in \eqref{eq:statACUSUM},\eqref{eq:stACUSUM} satisfies
$$
\Exp_\infty[\T]\geq e^{\nu}.
$$
Additionally, for the drift under the pre-change regime, we have $-\hat{\I}_\infty=\Exp_\infty[\log\frac{\f_0(\xi_t,\hth_{t-1})}{\f_\infty(\xi_t)}]<0$.
\end{lemma}

\begin{proof}
We compute the average false alarm period using similar ideas as in \cite[Lemma 8.2.1]{tartakovsky2014sequential}.
For $t>w$ let us define a Shiryaev-Roberts like statistic $\{\L_t\}$ through the recursion
$$
\L_t=(\L_{t-1}+1)\frac{\f_0(\xi_t,\hth_{t-1})}{\f_\infty(\xi_t)},~\L_w=0.
$$
Interestingly, $\{\L_t\}$ preserves the characteristic martingale property with respect to the $\Pro_\infty$ measure enjoyed by the classical Shiryaev-Roberts (SR) test statistic
$$
L_t=(L_{t-1}+1)\frac{\f_0(\xi_t,\theta)}{\f_\infty(\xi_t)},~L_0=0,
$$
even when we replace $\theta$ with the estimates $\{\hth_t\}$. Indeed we observe that
\begin{align*}
\Exp_\infty[\L_t-t|\cF_{t-1}] & =\Exp_\infty\left[(\L_{t-1}+1)\frac{\f_0(\xi_t,\hth_{t-1})}{\f_\infty(\xi_t)}-t\Big|\cF_{t-1}\right] =\L_{t-1}-(t-1).    
\end{align*}
The last equality is true because given $\cF_{t-1}$ we have $\hth_{t-1}$ fixed, therefore $f_0(\xi_t,\hth_{t-1})$ is a legitimate probability density for $\xi_t$. 
The martingale property of $\{\L_t-t\}$ and usage of Optional Sampling allows us to write for any stopping time $T$ with finite expectation that
\begin{equation}
\Exp_\infty[\L_T-T]=\Exp_\infty[\L_w-w]=-w \Rightarrow \Exp_\infty[T]-w=\Exp_\infty[\L_T].
\label{eq:A1}
\end{equation}

Let us now recall the WLCUSUM update in \eqref{eq:statACUSUM} which after exponentiation can be equivalently written as
$$
e^{\S_t}=\max\{e^{\S_{t-1}},1\}\frac{\f_0(\xi_t,\hth_{t-1})}{\f_\infty(\xi_t)},~e^{\S_w}=1.
$$
Using induction, the fact that $e^{\S_{w+1}}=\L_{w+1}$ and that for $x\geq0$,  $x+1\geq\max\{x,1\}$, it is straightforward to prove that for $t>w$ we have $\L_t\geq e^{\S_t}$. This suggests that the SR-statistic is larger than the exponential of the WLCUSUM statistic.
With this observation and using \eqref{eq:A1} we can now write
$$
e^\nu\leq \Exp_\infty[e^{\S_\T}]\leq\Exp_\infty[\L_{\T}]=\Exp_\infty[\T]-w\leq\Exp_\infty[\T],
$$
which proves the desired inequality.
We also observe that for any $\theta$ we have $\Exp_\infty[\log\frac{\f_0(\xi_1,\theta)}{\f_\infty(\xi_1)}]<0$, therefore $\Exp_\infty[\log\frac{\f_0(\xi_t,\hth_{t-1})}{\f_\infty(\xi_t)}|\cF_{t-1}]<0$, which in turn implies
$\Exp_\infty[\log\frac{\f_0(\xi_t,\hth_{t-1})}{\f_\infty(\xi_t)}]=-\hat{\I}_\infty<0$.
\end{proof}

Equating the lower bound $e^\nu$ provided by Lemma\,\ref{lem:2} to the desired average false alarm period $\gamma$, assures that the false alarm constraint $\Exp_\infty[\T]\geq\gamma$ is satisfied. Consequently, the threshold we select to use is equal to
\begin{equation}
\nu=\log\gamma.
\label{eq:gamma}
\end{equation}

Consider now the negative drift 
$-\hat{\I}_\infty$ mentioned in the lemma that appears under the pre-change regime. The drift under the $\Pro_\infty$ measure must be negative, because this assures restarts of the process and also that the average false alarm period will be an exponential function of the threshold. Since the drift is equal to $-\hat{\I}_\infty$, we have that $\hat{\I}_\infty$ must be positive. As we argued in the proof of Lemma\,\ref{lem:2} the approximate KL information number $\hat{\I}_\infty$ is indeed positive, and using \eqref{eq:assA4} we can study how the estimator affects this positive value. From \eqref{eq:assA4} we can see that
the first term of $\hat{\I}_\infty$ on the right-hand side is positive. One may wonder whether the second term due to the estimation error may also contribute to the positivity of $\hat{\I}_\infty$. Of course, the sign of this term is estimator-dependent. However, in the case of the MLE, it is easy to see that $-\F_\infty$ is the Hessian of $\Exp_\infty[\log\f_0(\xi_1,\theta)]$ evaluated at the point $\theta_\infty$ where this function is maximized\footnote{Provided of course that this value is not on the border of the allowable parameter set $\Theta$ and corresponds to an unconstrained optimizer.}. Consequently, $-\F_\infty$ is negative definite and therefore $\F_\infty$ positive definite. This means that the sign of $\mathrm{trace}\{\F_\infty\Sigma_\infty\}$ will be positive as well, contributing to the positivity of $\hat{\I}_\infty$ and therefore the negativity of the drift. We must, however, emphasize that other estimators do not necessarily share this desirable property and positivity is assured for sufficiently large window $w$.

The next step in our analysis consists in computing the worst-case average detection delay of \hbox{WLCUSUM}. In the following lemma, we present an important property for our detection strategy, which is also shared by the classical CUSUM and considerably facilitates the performance computation.
\begin{lemma}[Worst-case average detection delay]\label{lem:3}
For any changepoint $\tau\geq0$ we have that
\begin{align*}
&\mathrm{ess}\sup\Exp_\tau^\theta[\T-\tau|\T>\tau,\cF_\tau] \leq w + \Exp_\tau^\theta\Big[\Exp_\tau^\theta[(\T_\tau-\tau-w)\ind{\T_\tau>\tau+w}|\xi_{\tau+1}^{\tau+w}]\Big] =\Exp_0^\theta[\T].   
\end{align*}
\end{lemma}
\begin{proof}
Consider first a change at $\tau=0$. Since we start testing at $w+1$ we can write
$$
\Exp_0^\theta[\T]=w+\Exp_0^\theta\Big[\Exp_0^\theta[(\T-w)\ind{\T>w}|\xi_1^w]\Big].
$$
Suppose now that the change happens at $\tau$ then for $t>\tau+w$ it is clear that the test statistic $\S_t$ is larger than the test statistic $\S_{t,\tau}$ generated by \textit{starting} the WLCUSUM at time $\tau+w+1$ and using the samples $\{\xi_{\tau+w},\ldots,\xi_{\tau+1}\}$ to form the first estimate $\hth_{\tau+w}$.  
This suggests that if we use $\S_{t,\tau}$ instead of $\S_t$ in \eqref{eq:stACUSUM} then we will stop at a time $\T_\tau$ that satisties $\T_\tau\geq\T$. Clearly we also have that $\T_\tau$ is independent from $\cF_\tau$, because in $\S_{t,\tau}$ the data $\cF_\tau$ are not being used. With these observations in mind we can write
\begin{equation*}
\begin{aligned}
&\Exp_\tau^\theta[\T-\tau|\T>\tau,\cF_\tau] \\
& \leq w+\Exp_\tau^\theta[(\T-\tau-w)\ind{\T>\tau+w}|\T>\tau,\cF_\tau]\\
& =w+\Exp_\tau^\theta\Big[\Exp_\tau^\theta[(\T-\tau-w)\ind{\T>\tau+w}|\xi_{\tau+1}^{\tau+w},\cF_\tau]|\T>\tau,\cF_\tau\Big]\\
&\leq w+\Exp_\tau^\theta\Big[\Exp_\tau^\theta[({\T_\tau}\!-\!\tau \!-\!
w)\ind{{\T_\tau}>\tau+w}|\xi_{\tau+1}^{\tau+w},\! \cF_\tau]|\T>\tau,  
\cF_\tau\Big]\\
&=w+\Exp_\tau^\theta\Big[\Exp_\tau^\theta[({\T_\tau}-\tau-w)\ind{{\T_\tau}>\tau+w}|\xi_{\tau+1}^{\tau+w}]\Big]=\Exp_0^\theta[\T].
\end{aligned}    
\end{equation*}
The last equality is true because starting the procedure at $\tau+w+1$ when the change occurs at $\tau$ does not employ any information from $\cF_\tau$, consequently it is independent from $\cF_\tau$. Therefore, statistically this is the same as starting at $w+1$ with the change occurring at 0.
\end{proof}

The analysis of $\Exp_0^\theta[\T]$ is not simple due to the particular updating rule of $\S_t$. Since we are only interested in finding a suitable upper bound, we are going to introduce an alternative stopping time $\T'$, which is easier to analyze. Its definition and connection to $\T$ are presented in the following lemma.
\begin{lemma}\label{lem:new1}
Assume that the observations follow the post-change regime and that $\hat{\I}_0>0$. For $t>w$ define the process $\{\U_t\}$ using the recursion
\begin{equation}
\U_t=\U_{t-1}+\log\frac{\f_0(\xi_t,\hth_{t-1})}{\f_\infty(\xi_t)},~~\U_w=0,
\label{eq:statU}
\end{equation}
and the stopping time
$$
\T'=\inf\{t>w:\U_t\geq\nu\},
$$
then $\T'$ stops a.s. and we have $\T'\geq\T$. 
\end{lemma}
\begin{proof}
Because $\S_{w+1}=\U_{w+1}=\log\frac{\f_0(\xi_{w+1},\hth_{w})}{\f_\infty(\xi_{w+1})}$ and $x^+=\max\{x,0\}\geq x$, using induction it is simple to prove that $\S_t\geq\U_t$ for $t>w$. This of course implies that $\U_t$ will require more time than $\S_t$ to reach the same threshold $\nu$, which means that $\T'\geq\T$. Consequently, $\Exp_0^\theta[\T']\geq\Exp_0^\theta[\T]$. The fact that $\T'$ will stop a.s. is guaranteed by the positivity of $\hat{\I}_0=\Exp_0^\theta[\log\frac{\f_0(\xi_t,\hth_{t-1})}{\f_\infty(\xi_t)}]$.
\end{proof}

Finding a bound for $\Exp_0^\theta[\T']$ is simpler because as we see from its definition in \eqref{eq:statU} we have that $\U_t$ is a sum of stationary (but $w$-dependent) terms. The estimate we are looking for is provided in the next theorem which identifies, with the help of Lemma\,\ref{lem:new1}, an upper bound for the performance of the WLCUSUM test.
\begin{theorem}[WADD of WLCUSUM]\label{th:1}
Assume $\hat{\I}_0>0$. We have the following upper bound for the worst-case performance of WLCUSUM
\begin{equation}
\begin{aligned}
\Exp_0^\theta[\T]& \leq\Exp_0^\theta[\T'] \leq\frac{\log\gamma+\frac{\hat{\!\J}_0}{\hat{\I}_0}+\Big(\frac{\hat{\!\J}_0}{\hat{\I}_0}\log\gamma\Big)^{\half}+w\I_0+\Big(\frac{\hat{\!\J}_0}{\hat{\I}_0}\I_0w\Big)^{\half}}{\hat{\I}_0},    
\end{aligned}
\label{eq:th1}
\end{equation}
where $\hat{\I}_0$ and $\hat{\!\J}_0$ were defined in Assumption\,A4.
\end{theorem}
\begin{proof}
The complete proof of our claim involves several steps. We provide a proof sketch in three steps. First, from Lemma \ref{lem:3} we have that the worst-case average detection delay equals to $\Exp_0^\theta[\T]$, thus we only need to bound $\Exp_0^\theta[\T]$. Second, we apply the Wald's identity for $w$-dependent samples to bound the expected stopping time $\Exp_0^\theta[\T']$ for process $\{\U_t\}$ in \eqref{eq:statU}, this is an upper bound for the detection delay $\Exp_0^\theta[\T]$ based on Lemma \ref{lem:new1}. Finally, we substitute the threshold $\nu=\log\gamma$ according to Lemma \ref{lem:2} to ensure the ARL satisfies $\Exp_\infty[\T]\geq \gamma$.
Details can be found in the Appendix. 
\end{proof}

\begin{remark}\label{rem:3}
If we fix the window size $w$ then from Theorem\,\ref{th:1} we have that the upper bound of the worst-case average detection delay of WLCUSUM can be written as
$$
\Exp_0^\theta[\T]\leq\frac{\log\gamma}{\hat{\I}_0}\big(1+o(1)\big).
$$
Because from \eqref{eq:inequ} we have $\hat{\I}_0<\I_0$ this means that we cannot assure first-order asymptotic optimality with  fixed $w$.
From \eqref{eq:assA4} we see that the difference between the optimum denominator $\I_0$ (enjoyed by the exact CUSUM) and $\hat{\I}_0$ obtained by our proposed stopping time is (asymptotically) equal to $(1/w)\mathrm{trace}\{\Sigma_0\F_0\}$. As we argued in \eqref{eq:CRLB} the factor $\mathrm{trace}\{\Sigma_0\F_0\}$ due to the Cramer-Rao Lower Bound cannot become smaller than the length $K$ of the parameter vector, which corresponds to the best possible performance since it maximizes the denominator $\hat{\I}_0$. We recall that this optimal value is attained asymptotically by the MLE. 
\end{remark}

In order for WLCUSUM to be first-order asymptotically optimum, it is necessary to increase $w$ with the average run length constraint $\gamma$ in a manner that guarantees that $w$ remains negligible compared to the optimum CUSUM performance but grows sufficiently fast so that it provides efficient estimates of $\theta$. Since the CUSUM detection delay increases as $\log\gamma$ (see Lemma\,\ref{lem:1}) this means that we must select $w=o(\log\gamma)$. At the same time, to obtain improved estimates, we must let $w\to\infty$ as $\gamma\to\infty$. These observations will be taken into account to decide the appropriate growth rate of $w$.

From Lemma\,\ref{lem:1} we deduce that the classical CUSUM stopping time $\Tc$ satisfies
$$
\Exp_0^\theta[\Tc]/\frac{\log\gamma}{\I_0}=1+\varTheta\left(\frac{1}{\log\gamma}\right).
$$
Since CUSUM is (strictly) optimum \cite{mous-astat-1986}, we have
$$
\Exp_0^\theta[\Tc]/\frac{\log\gamma}{\I_0}\leq\Exp_0^\theta[\T]/\frac{\log\gamma}{\I_0}.
$$
Using the results of Theorem\,\ref{th:1} and properly selecting the increase rate of $w$, in the next theorem, we demonstrate that WLCUSUM enjoys the desired first-order asymptotic optimality.


\begin{theorem}[Asymptotic optimality of WLCUSUM]\label{lem:last}

If the window size $w$ satisfies $w\rightarrow\infty$ as $\gamma\rightarrow\infty$ with $w=o(\log\gamma)$, then
$$
\Exp_0^\theta[\T]/\frac{\log\gamma}{\I_0}\leq 1+\varTheta\left(\frac{1}{\sqrt{\log\gamma}}\right) + \varTheta\left(\frac{1}{w}\right) + \varTheta\left(\frac{w}{\log\gamma}\right).
$$
The maximal convergence speed to unity is $1+\varTheta\left(\frac{1}{\sqrt{\log\gamma}}\right)$ and is achieved when $w=\varTheta\big(\sqrt{\log\gamma}\big)$.
\end{theorem}
\begin{proof}
Using \eqref{eq:assA4} we note that the ratio $\hat{\!\J}_0/\hat{\I}_0$ in the numerator in \eqref{eq:th1} is a function of $w$, which for sufficiently large $w$ can be bounded by a constant that does not depend on $w$. If we select $w=o(\log\gamma)$ but $w\rightarrow\infty$ as $\gamma\rightarrow\infty$, replace it in the upper bound in \eqref{eq:th1} and also use \eqref{eq:assA4} for the denominator then we can write 
\begin{equation}
\Exp_0^\theta[\T]/\frac{\log\gamma}{\I_0}\leq\frac{1+\varTheta(\frac{1}{\log\gamma})+\varTheta(\frac{1}{\sqrt{\log\gamma}})+\varTheta(\frac{w}{\log\gamma})+\varTheta(\frac{w^{1/2}}{\log\gamma})}{1+\varTheta(\frac{1}{w})}.
\label{eq:mammamia}
\end{equation}
All $\varTheta(\cdot)$ terms in the right-hand side converge to 0 and the overall convergence rate is dominated by the slowest term among: $
\big\{\varTheta(\frac{1}{w}), \varTheta(\frac{1}{\sqrt{\log\gamma}}), \varTheta(\frac{w}{\log\gamma}) \big\}
$. Note that the best possible convergence rate towards 1 is obtained when we select $w=\varTheta\big(\sqrt{\log\gamma}\big)$. Therefore $w \sim \sqrt{\log\gamma}$ is the optimal choice but any window size $w$ satisfying $w\rightarrow\infty$ as $\gamma\rightarrow\infty$ with  $w=o(\log\gamma)$ is sufficient to guarantee the first-order optimality.
\end{proof}

The immediate consequence of Theorem\,\ref{lem:last} is the asymptotic optimality of WLCUSUM since
$$
\begin{aligned}
&1=\lim_{\gamma\to\infty}\Exp_0^\theta[\Tc]/\frac{\log\gamma}{\I_0}\leq\lim_{\gamma\to\infty}\Exp_0^\theta[\T]/\frac{\log\gamma}{\I_0} \\
& \leq\lim_{\gamma\to\infty}\left\{1+\varTheta\left(\frac{1}{\sqrt{\log\gamma}}\right) + \varTheta\left(\frac{1}{w}\right) + \varTheta\left(\frac{w}{\log\gamma}\right)\right\}=1,    
\end{aligned}
$$
when selecting $w=o(\log\gamma)$ with $w\rightarrow\infty$ as $\gamma\rightarrow\infty$, 
which proves that $\Exp_0^\theta[\T]/\frac{\log\gamma}{\I_0}$ attains the same nonzero limit as the optimum $\Exp_0^\theta[\Tc]/\frac{\log\gamma}{\I_0}$ and is therefore asymptotically optimum as well. As we have indicated, the optimal convergence rate towards 1 is of the form $\varTheta\big((\log\gamma)^{-\half}\big)$ using the optimum window size $w=\varTheta\big(\sqrt{\log\gamma}\big)$. Although there is a definite loss in performance as compared to the rate of the exact CUSUM which is $\varTheta\big((\log\gamma)^{-1}\big)$ (Lemma\,\ref{lem:1}), our scheme still enjoys asymptotic optimality.

\begin{remark}\label{rem:5.5}
One may claim that the specific rate we obtained is a consequence of the crude upper bound $O(\sqrt{\nu})$ we established for the average overshoot (see Appendix, proof of Theorem\,\ref{th:1}), and we could have obtained better results if this bound were a constant as in the i.i.d.~case \cite{lorden1970excess}. In fact a constant would have resulted in a term of the form $\varTheta(\frac{1}{\log\gamma})$ in place of the $\varTheta(\frac{1}{\sqrt{\log\gamma}})$ we have now in the numerator of \eqref{eq:mammamia}. For the computation of the convergence rate, this would have led to the smallest power satisfying $\beta=\min\{\alpha,1-\alpha\}$ instead of $\beta=\min\{\alpha,1-\alpha,\frac12\}$. However, if we maximize $\min\{\alpha,1-\alpha\}$ we still obtain that the best power rate is equal to $\frac12$ which is achieved when $\alpha=\frac12$ or equivalently $w=\varTheta(\sqrt{\log\gamma})$. 
\end{remark}

\begin{remark}[Comparison with window-limited GLR]\label{rem:wGLR} In order for the window-limited GLR procedure in \eqref{glr_lai} to enjoy first-order asymptotic optimality, the window size $w$ must also grow with $\gamma$ in a certain rate. In particular $w$ must satisfy $\lim\inf w/\log\gamma > \I_0^{-1}$ and $\log w = o(\log\gamma)$ to ensure effective detection \cite{lai-ieeetit-1998}, namely window at least as large as the CUSUM average detection delay. In the proposed WLCUSUM procedure, we only require a window size $w=o(\log\gamma)$ with $w\rightarrow\infty$ as $\gamma\rightarrow\infty$, which can be much smaller, even negligible, than the requirement in the window-limited GLR. 
\end{remark}

In addition to the optimal order $w=\varTheta\big(\sqrt{\log\gamma}\big)$ obtained in Theorem\,\ref{lem:last}, we also provide a detailed characterization of the constant terms, which will be helpful for choosing the appropriate window size in practice. It should be mentioned that when the post-change quantities are unknown in practice, we can use the parallel variant in Section\,\ref{sec:variants} instead.

\begin{lemma}[Optimal window size]\label{lem:w_constant}
The optimal window size that minimizes the upper bound of the WADD of WLCUSUM given in \eqref{eq:th1} is 
$$
w=\frac{\sqrt{\mathrm{trace}\{\Sigma_0\F_0\}}}{\I_0\sqrt{2}}\sqrt{\log\gamma}\left\{1+\varTheta\left(\frac{1}{(\log\gamma)^{\frac{1}{4}}}\right)\\\right\}.
$$ 
\end{lemma}
\begin{proof}
Substitute \eqref{eq:assA4} for $\hat{\I}_0$ and $\hat{\!\J}_0$ in \eqref{eq:th1} and note that the upper bound in \eqref{eq:th1} can be written as 
\[
\frac{\log\gamma}{\I_0} \Bigg\{ 1 + \left(\frac{\J_0}{\I_0}\right)^{\half}\frac{1}{(\log\gamma)^{\half}}+\frac{\I_0}{\log\gamma}w  + \frac{\mathrm{trace}\{\Sigma_0\F_0\}}{\I_0}\frac{1}{2w}  +
\varTheta\left(
\frac{w^\half}{\log\gamma}
\right)+\varTheta\left(\frac{1}{w^2}\right)
\Bigg\}.    
\]
Given from Theorem\,\ref{lem:last} that the optimum $w$ we are seeking satisfies $w=\varTheta(\sqrt{\log\gamma})$, the previous expression can be further simplified as follows
\begin{align*}
&\frac{\log\gamma}{\I_0} \Bigg\{ 1 + \left(\frac{\J_0}{\I_0}\right)^{\half}\frac{1}{(\log\gamma)^{\half}}+\frac{\I_0}{\log\gamma}w  + \frac{\mathrm{trace}\{\Sigma_0\F_0\}}{\I_0}\frac{1}{2w}  +
w^\half\varTheta\left(\frac{1}{\log\gamma}
\right)\Bigg\}\\
& \leq \frac{\log\gamma}{\I_0} \Bigg\{ 1 + \left(\frac{\J_0}{\I_0}\right)^{\half}\frac{1}{(\log\gamma)^{\half}}+\frac{\I_0}{\log\gamma}w  + \frac{\mathrm{trace}\{\Sigma_0\F_0\}}{\I_0}\frac{1}{2w}+
\frac{w^\half D}{\log\gamma}
\Bigg\},
\end{align*}
for some constant $D$ coming from the definition of $\varTheta(\cdot)$. Minimizing the last expression over $w$ yields the following equation for the derivative
$$
\frac{\I_0}{\log\gamma} - \frac{\mathrm{trace}\{\Sigma_0\F_0\}}{\I_0}\frac{1}{2w^2}+
\frac{1}{2w^\half}\frac{D}{\log\gamma}
=0.
$$
Finding a more precise form for the optimum $w$ since we know its order of magnitude, amounts to selecting $w=C\sqrt{\log\gamma}\big(1+\varTheta\big((\log\gamma)^{-\rho}\big)\big),\rho>0$ and specifying $C$ and $\rho$. Substituting into the previous equation we obtain
\begin{align*}
&\frac{\I_0}{\log\gamma} - \frac{\mathrm{trace}\{\Sigma_0\F_0\}}{2\I_0  C^2\log\gamma}\left(1-\varTheta\left(\!\frac{1}{(\log\gamma)^{\rho}}\!\right)\right)+
\varTheta\left(\!\frac{1}{(\log\gamma)^{\frac{5}{4}}}
\!\right)\\
& =\frac{\I_0}{\log\gamma} - \frac{\mathrm{trace}\{\Sigma_0\F_0\}}{2\I_0 C^2\log\gamma}+
\varTheta\left(\!\frac{1}{(\log\gamma)^{\frac{5}{4}}}\!\right)-\varTheta\left(\!\frac{1}{(\log\gamma)^{1+\rho}}\!\right)=0,
\end{align*}
where for $|\epsilon|\ll1$ we used the approximation $1/(1+\epsilon)^2\approx1-2\epsilon$.
As we can see there are two terms which are of different order of magnitude therefore we need to make them 0 separately. This means
$$
\begin{aligned}
&\frac{\I_0}{\log\gamma} - \frac{\mathrm{trace}\{\Sigma_0\F_0\}}{\I_0}\frac{1}{2C^2\log\gamma}=0, \ \text{and } \ \varTheta\left(\frac{1}{(\log\gamma)^{\frac{5}{4}}}\right)-\varTheta\left(\frac{1}{(\log\gamma)^{1+\rho}}\right)=0.    
\end{aligned}
$$
The first results in $C=\sqrt{\mathrm{trace}\{\Sigma_0\F_0\}}/\I_0\sqrt{2}$ and the second in $\rho=\frac{1}{4}$. This completes the proof of the lemma.
\end{proof}


\section{Parallel WLCUSUM}\label{sec:variants}

Although the proposed WLCUSUM procedure in \eqref{eq:statACUSUM} and \eqref{eq:stACUSUM} enjoys first-order optimality, there are some practical concerns: (i) for any fixed $w$, the WLCUSUM procedure will start after observing $w$ samples, which leads to a delay at least $w$; (ii) the optimal window size derived in Lemma\,\ref{lem:w_constant} might be unknown when the post-change signal strength $\I_0$ is unknown; of course we can impose a pre-set lower bound on $\I_0$ and design the window size according to this lower bound, but this may lead to larger than necessary detection delays. In this section, we provide a variant to WLCUSUM that can resolve the above issue being also suitable for practical implementation. 

In case the optimal window size is difficult to estimate beforehand, we propose to perform \textit{parallel} WLCUSUMs with different window sizes. 
More specifically, we run in parallel multiple WLCUSUMs with window sizes ranging from 1 up to some maximal value $W$. For each window size $w$, we denote the corresponding test statistic with $\S_t(w)$. At each time instant $t$, all statistics are compared to a common threshold $\nu$ and the first time $\T_\P$ any of the statistics hits or exceeds $\nu$ is the time the parallel WLCUSUM will stop. If to each statistic we associate the corresponding stopping time $\T(w)$, then it is also true that the parallel WLCUSUM satisfies 
\begin{equation}
\T_\P =\min_{1\leq w\leq W}\{\T(w)\}.
\label{eq:TpvsTw}
\end{equation}
In adopting this approach, we do not wait for $W$ samples in order to start the detection procedure. Indeed the first WLCUSUM with $w=1$ has to wait for only one sample and then perform tests continuously. Furthermore, we do not need to specify an exact window size beforehand. In a sense by running in parallel multiple windows, it is mostly the best window that tends to be the first to stop.

The maximal window $W$ does not have to be different from what we estimated in our previous analysis. Also, the average detection delay at $\tau=0$ is still the worst case compared with any other change time $\tau>0$, i.e., the analysis in Lemma\,\ref{lem:3} can be modified to cover $\T_\P$ as well and we can easily demonstrate that the parallel WLCUSUM, with a suitable maximal window size $W$, is also first-order asymptotically optimal matching the performance of WLCUSUM with the optimal window. We provide the necessary proof of this claim in the next lemma and the discussion that follows.

\begin{lemma}\label{lem:PWLCUSUM}
The average run length of the Parallel WLCUSUM with maximum window size $W$ satisfies 
\[
\Exp_\infty[\T_\P]\geq \frac{1}{W} e^{\nu}.
\]
\end{lemma}
\begin{proof}
Recall that $\S_t(w)$ denotes the statistic of the WLCUSUM with window size $w$. Similarly to Lemma\,\ref{lem:2}, let us define a Shiryaev-Roberts like statistic $\{\L_t(w)\}$, under a window size $w$, through the recursion
$$
\L_t(w)=\big(\L_{t-1}(w)+1\big)\frac{\f_0(\xi_t,\hth_{t-1})}{\f_\infty(\xi_t)},~\L_w(w)=0, \ t \geq w+1.
$$
Following the arguments in Lemma\,\ref{lem:2}, we have $\L_t(w)\geq e^{\S_t(w)}$ for $t>w$ and $\{\L_t(w)\}$ preserves the characteristic martingale property with respect to the $\Pro_\infty$ measure. Applying Optional Sampling, we have 
$\Exp_\infty[\L_T(w)-T]=-w$ for any stopping time $T$ with finite expectation. The parallel \hbox{WLCUSUM} statistic can be written as $\max_{1\leq w\leq W}\{\S_t(w)\}$ with the corresponding stopping time satisfying $\T_\P=\inf\{t>0: \max_{1\leq w\leq W}\{\S_t(w)\} \geq \nu\}$. For the stopping time we have
$$
e^\nu \leq e^{\max\limits_{1\leq w\leq W}\{\S_{\T_\P}(w)\}}\leq \max_{1\leq w\leq W}\{\L_{\T_\P}(w)\}\leq \sum_{w=1}^W \L_{\T_\P}(w).
$$
Taking expectations under $\Pro_\infty$ on both sides yields
\[
e^\nu \leq \sum_{w=1}^W (\Exp_\infty[\T_\P]-w) 
< W \Exp_\infty[\T_\P],
\]
which proves the desired inequality.
\end{proof}

Equating the lower bound $\frac{1}{W}e^\nu$ provided by Lemma\,\ref{lem:PWLCUSUM} to the desired average false alarm period $\gamma>1$ assures that the false alarm constraint $\Exp_\infty[\T_\P]\geq\gamma$ is satisfied. Consequently, the threshold we select to use is equal to
\begin{equation}
\nu=\log(W\gamma)
\label{eq:gamma-P}
\end{equation}
as opposed to $\log\gamma$ we use for a single window.


For the average detection delay, note that the delay of the parallel WLCUSUM will certainly be no larger than the delay of each individual WLCUSUM procedure. Therefore, using \eqref{eq:TpvsTw} and Theorem\,\ref{th:1} we can write 
\[
\begin{aligned}
& \Exp_0^\theta[\T_\P]  \leq \min_{1\leq w\leq W}\Exp_0^\theta[\T(w)] \leq \!\! \min_{1\leq w\leq W} \! \frac{\log (W\gamma) \! + \! \frac{\hat{\!\J}_0}{\hat{\I}_0} 
\! + \! \Big(\frac{\hat{\!\J}_0}{\hat{\I}_0}\log(W\gamma)\Big)^{\half}
\!+ \! w\I_0 \!+\! \Big(\frac{\hat{\!\J}_0}{\hat{\I}_0}\I_0w\Big)^{\half}}{\hat{\I}_0},    
\end{aligned}
\]
where $\hat{\!\J}_0,\hat{\I}_0$ are the estimated information numbers under window size $w$. Comparing the previous upper bound to the one obtained in \eqref{eq:th1} we can see that if  $\log W = o(\sqrt{\log\gamma})$ the two upper bounds are asymptotically equivalent for each $w$. This implies that the parallel WLCUSUM performs the optimization with respect to $w$ \textit{automatically} without the need of prior manual selection. Of course this is true as long as $W$ exceeds the value of the optimum window and at the same time conforms with the requirement $\log W = o(\sqrt{\log\gamma})$. Regarding the last constraint it is trivially satisfied by the window sizes we employ namely $W=\varTheta\big((\log\gamma)^\alpha\big)$. Indeed if we consider $W=\varTheta\big((\log\gamma)^\alpha\big)$ 
with $1>\alpha>\half$ then $\log W=\varTheta(\log\log\gamma)=o(\sqrt{\log\gamma})$ but also $W$ exceeds the optimum window size since the latter satisfies $\varTheta\big(\sqrt{\log\gamma}\big)$.

\begin{remark}
\label{rem:computation}
The computational complexities 
of WLCUSUM, parallel WLCUSUM, and window-limited GLR can be ordered as follows:
\[
\text{WLCUSUM}<\text{parallel WLCUSUM}<\text{window-limited GLR}.
\]
Although all three procedures with the appropriate choice of window size are first-order asymptotically optimal, they differ in the number of operations needed per update. To form a more precise idea about the requirements in computation, let us assume that the estimation process needs constant complexity per estimate when considering estimates of windows of consecutive sizes. This is because it can compute each estimate by updating the estimate of the previous size.
This is for example the case when computing arithmetic averages of samples of increasing windows since these estimates can be performed recursively. For the same estimation problem if we are interested only in a single window then the complexity of the estimator is proportional to the window size. Of course complexity can be significantly higher when estimates need to be obtained through iterative solutions. 

Let us start with WLCUSUM with window size $w$. Its complexity is $\varTheta(w)$ since we compute the estimate with complexity $\varTheta(w)$ and then update the corresponding statistic with a constant number of operations. The parallel WLCUSUM with maximal window $W$ will require $\varTheta(W)$ operations to compute all the estimates of \textit{consecutive} window sizes and $\varTheta(W)$ operations to update the corresponding statistics (constant number of operations per statistic), therefore the total complexity is still $\varTheta(W)$. In the case of the window-limited GLR if $w$ is the size of the adopted window then we need $\varTheta(w)$ to compute the estimates of the consecutive smaller window sizes. Then, each such estimate is applied to all the samples in the smaller window to form an approximation of the corresponding log-likelihood ratio. The total complexity required by these updates is $\varTheta(1)+\cdots+\varTheta(w)=\varTheta(w^2)$. The overall complexity for estimation and updates is $\varTheta(w)+\varTheta(w^2)=\varTheta(w^2)$.

For WLCUSUM if we consider the optimal window size $w=\varTheta(\sqrt{\log\gamma})$, this results in complexity $\varTheta(\sqrt{\log\gamma})$ per time update. For the parallel WLCUSUM with a maximal window size $W=\varTheta\big((\log\gamma)^\alpha\big)$ and $1>\alpha>\half$ to cover the optimal window size, we need complexity $\varTheta\big((\log\gamma)^\alpha\big)$. In the case of the window-limited GLR we know that $w$ must be at least of the same order as the detection delay therefore $w=\varTheta(\log\gamma)$, suggesting that the corresponding computational complexity is $\varTheta\big((\log\gamma)^2\big)$. As we can see the difference in number of operations between the two versions of WLCUSUM and the window-limited GLR is significant.
\end{remark}

\section{Experiments}\label{sec:numerical}


Simulation studies are performed to compare WLCUSUM and parallel WLCUSUM to the exact CUSUM and the window-limited GLR. In the exact CUSUM all parameters are considered known and the corresponding detection delay is optimum \cite{mous-astat-1986}. In all experiments, the average detection delay and average false alarm period are obtained using direct estimation and averaging over 1000 independent trials.


\subsection{Univariate Normal Mean-Shift}
We consider a normal mean-shift example. Let $\f_\infty(\cdot)$ be a canonical Gaussian pdf $\N(0,I_K)$ and $\f_0(\cdot,\theta)$ a Gaussian $\N(\theta,I_K)$. Since $\f_0(\cdot,\theta)$ with $\theta=0$ is equal to $\f_\infty(\cdot)$ we impose a constraint on $\theta$ in order to avoid the two densities becoming equal. For the mean $\theta$ we assume that $\theta\in\Theta=\{\theta:\left\Vert\theta\right\Vert\geq\vartheta>0\}$ where the \textit{barrier} value $\vartheta$ is considered known and can be interpreted as the smallest signal strength that we are interested in detecting. 
For our estimate of the mean we employ the MLE with the solution projected onto the allowable parameter set $\Theta$. Specifically
$$
\hth_t=\frac{\bar{\theta}_t}{\left\Vert \bar{\theta}_t \right\Vert}\max\{\left\Vert\bar{\theta}_t\right\Vert,\vartheta\},~~\text{where}~\bar{\theta}_t=\frac{1}{w}\sum_{s=0}^{w-1}\xi_{t-s}.
$$

We first present a suite of numerical results for the one-dimensional ($K=1$) case with varying signal strengths. In the first setting, we set $\theta = 1$ and $\vartheta=0.5$.  Figure\,\ref{fig:edd1_final}(a) depicts the worst-case expected detection delay versus the average false alarm period (ARL), where the black line corresponds to the exact CUSUM procedure which lies below all the other curves. 
\begin{figure}[ht!]
\begin{tabular}{cc}
\includegraphics[width = 0.45\linewidth]{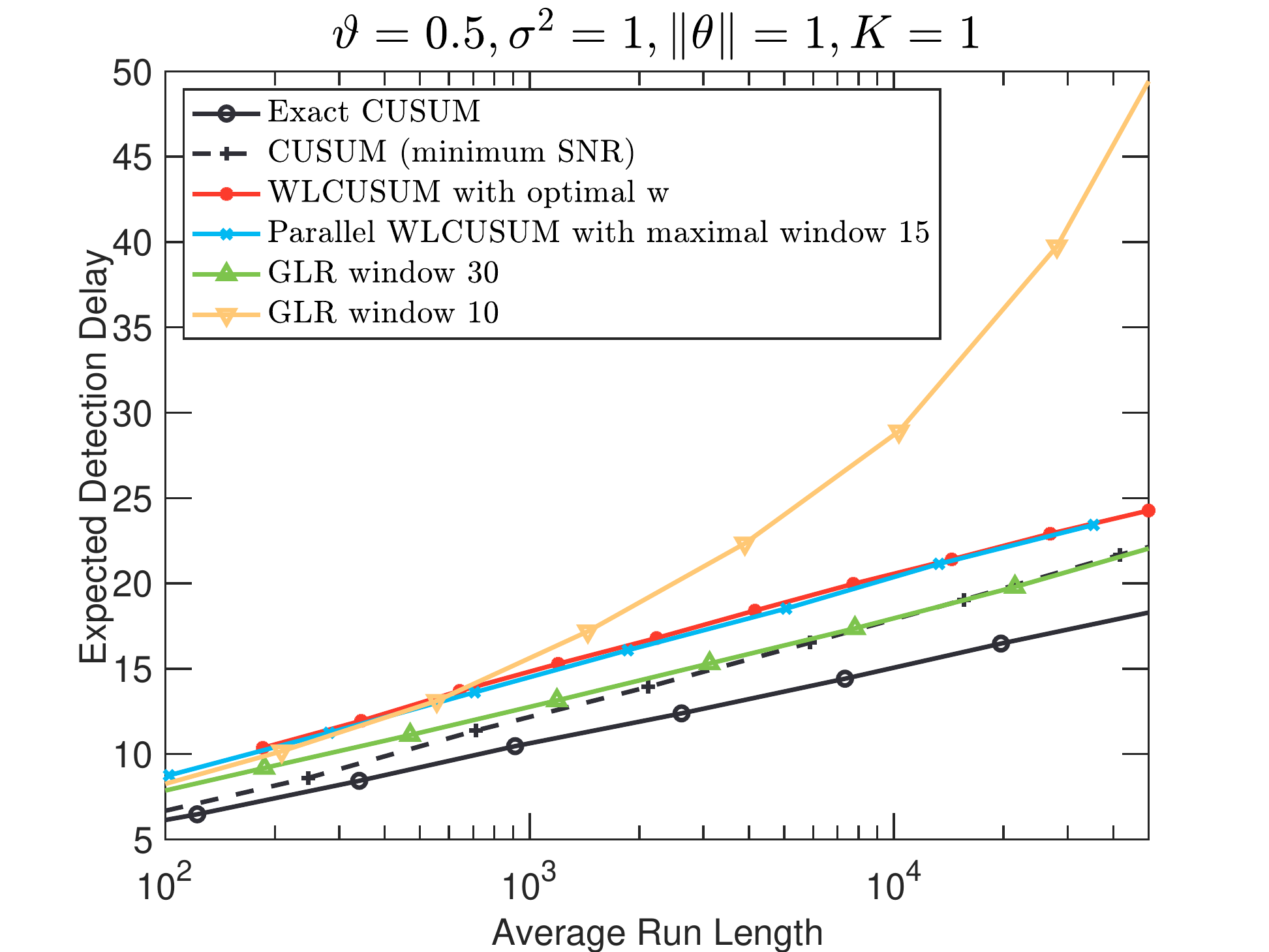} & \includegraphics[width = 0.45\linewidth]{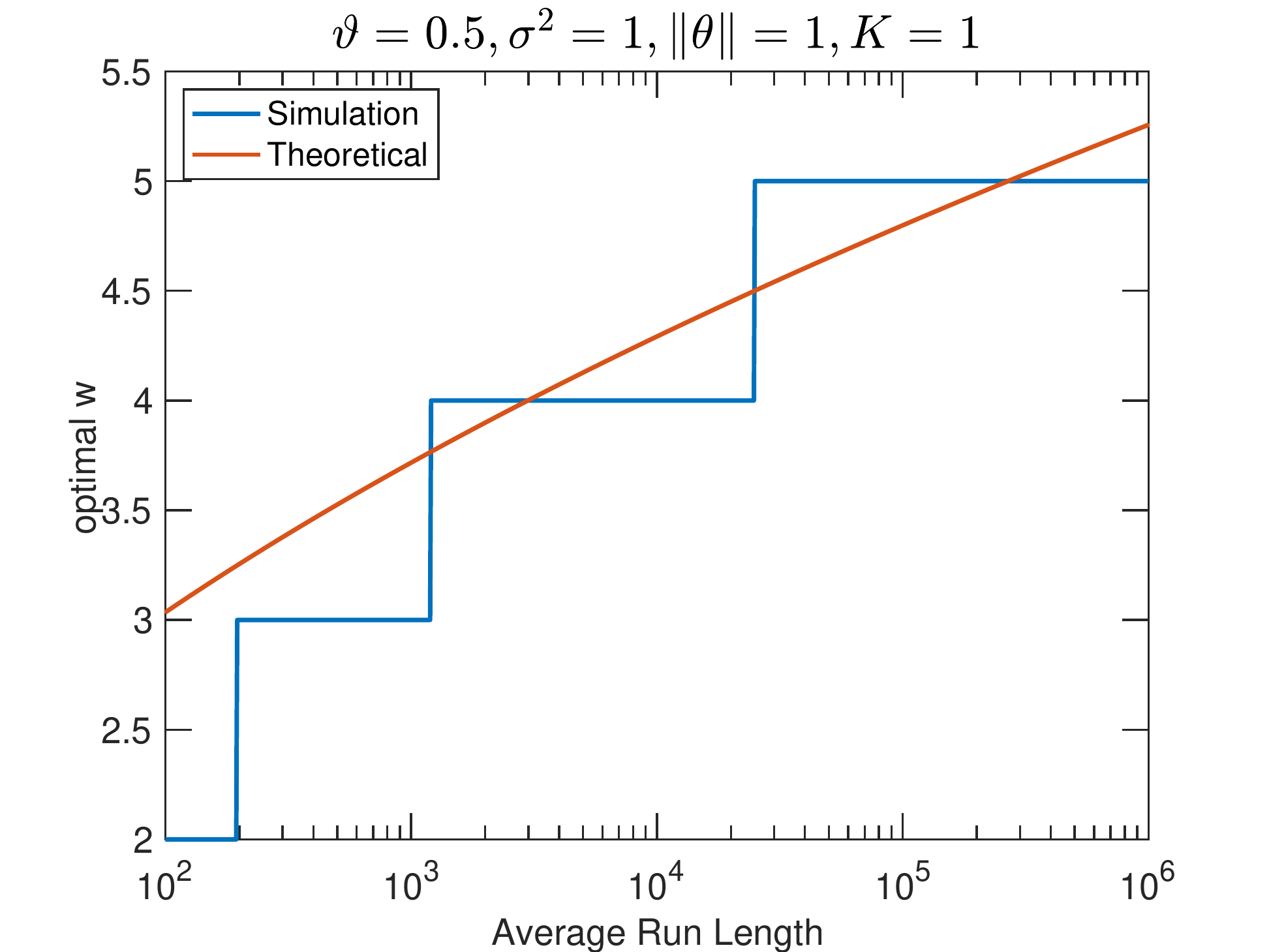}\\
(a) & (b) \\ 
\end{tabular}
\caption{(a) Detection delays for univariate normal mean-shift from 0 to $\theta=1$, barrier $\vartheta=0.5$; (b) Corresponding optimal window size as a function of ARL for WLCUSUM.}
\label{fig:edd1_final}
\end{figure}
The red line corresponds to \hbox{WLCUSUM} with optimal window size with Figure\,\ref{fig:edd1_final}(b) depicting how the optimal window must change as a function of ARL (theoretical and simulations). In Figure\,\ref{fig:edd1_final}(a) with blue we can also see the  parallel WLCUSUM procedure with maximal window size $W=15$. For comparison we also present in green the window-limited GLR defined in \eqref{glr_lai} with a window size of 30 (chosen to be larger than $\log\text{ARL}/\I_0$ as explained in Remark\,\ref{rem:wGLR}) and in yellow with an insufficient window size 10. In the latter we notice the remarkable performance degradation if we do not use the appropriate window. Regarding our proposed methods we observe that the parallel WLCUSUM has nearly similar performance as the WLCUSUM with optimal window size, without requiring any prior specification of the optimal window. The window-limited GLR (with sufficient window size), in this scenario exhibits a smaller detection delay than both versions of WLCUSUM but at the expense of a higher computational cost as explained in Remark\,\ref{rem:computation}.
Finally, we compare all procedures with the CUSUM test under minimal signal strength in black dashed line, i.e., we calculate the CUSUM test using densities of $\N(0.5,1)$ and $\N(0,1)$.

For the optimal window size depicted in Figure\,\ref{fig:edd1_final}(b) we simulated sizes varying from 1 to 15 and for each ARL value we selected the size with the smallest detection delay. This is depicted by the blue curve.
With red we have plotted the theoretical optimal value obtained in Lemma\,\ref{lem:w_constant}. We observe a relatively good agreement between the two possibilities for large ARL values which is to be expected since our theoretical results are asymptotic.
%


\begin{figure}[ht!]
\centering
\begin{tabular}{cc}
\includegraphics[width = 0.45\linewidth]{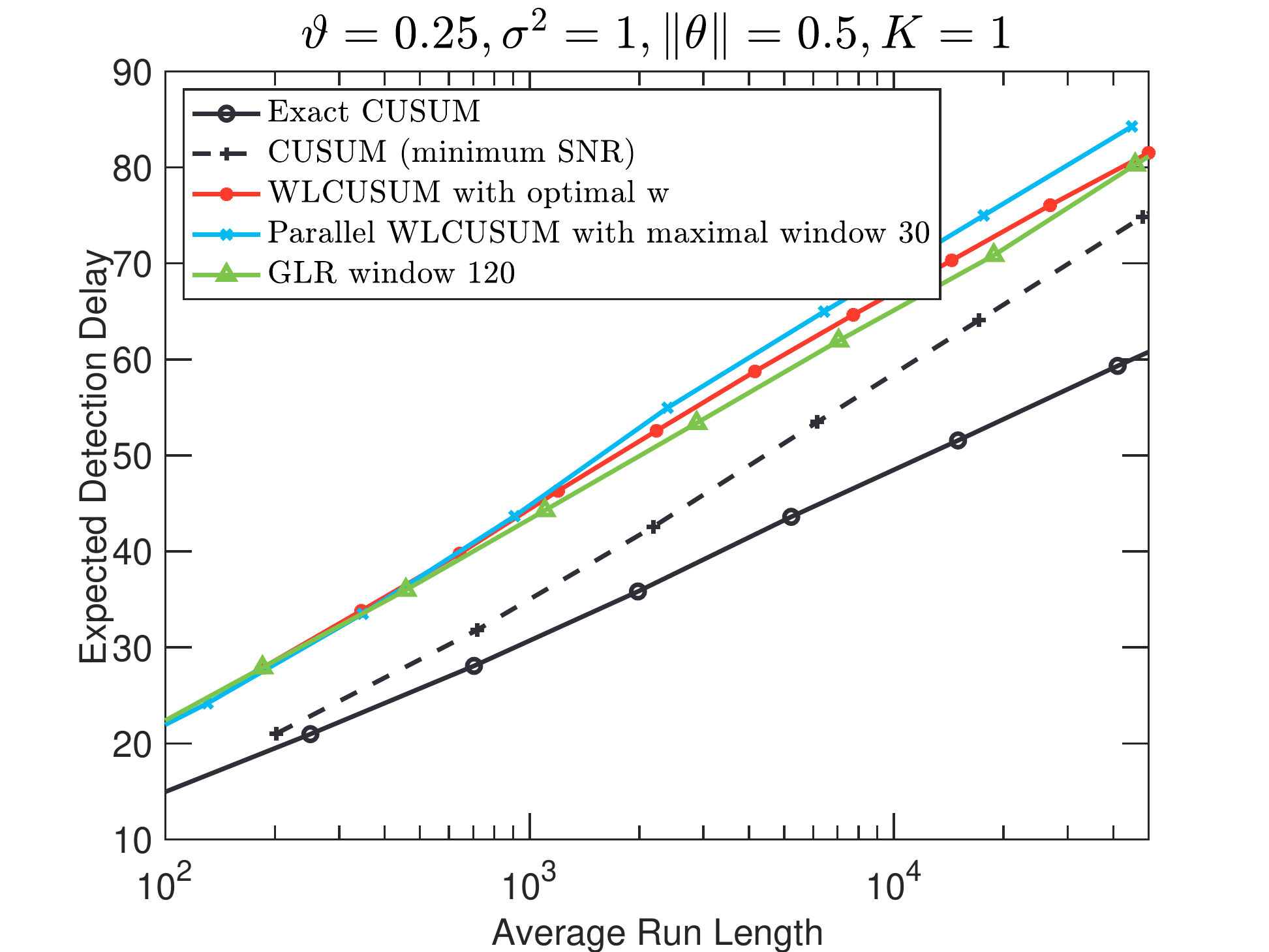} & \includegraphics[width = 0.45\linewidth]{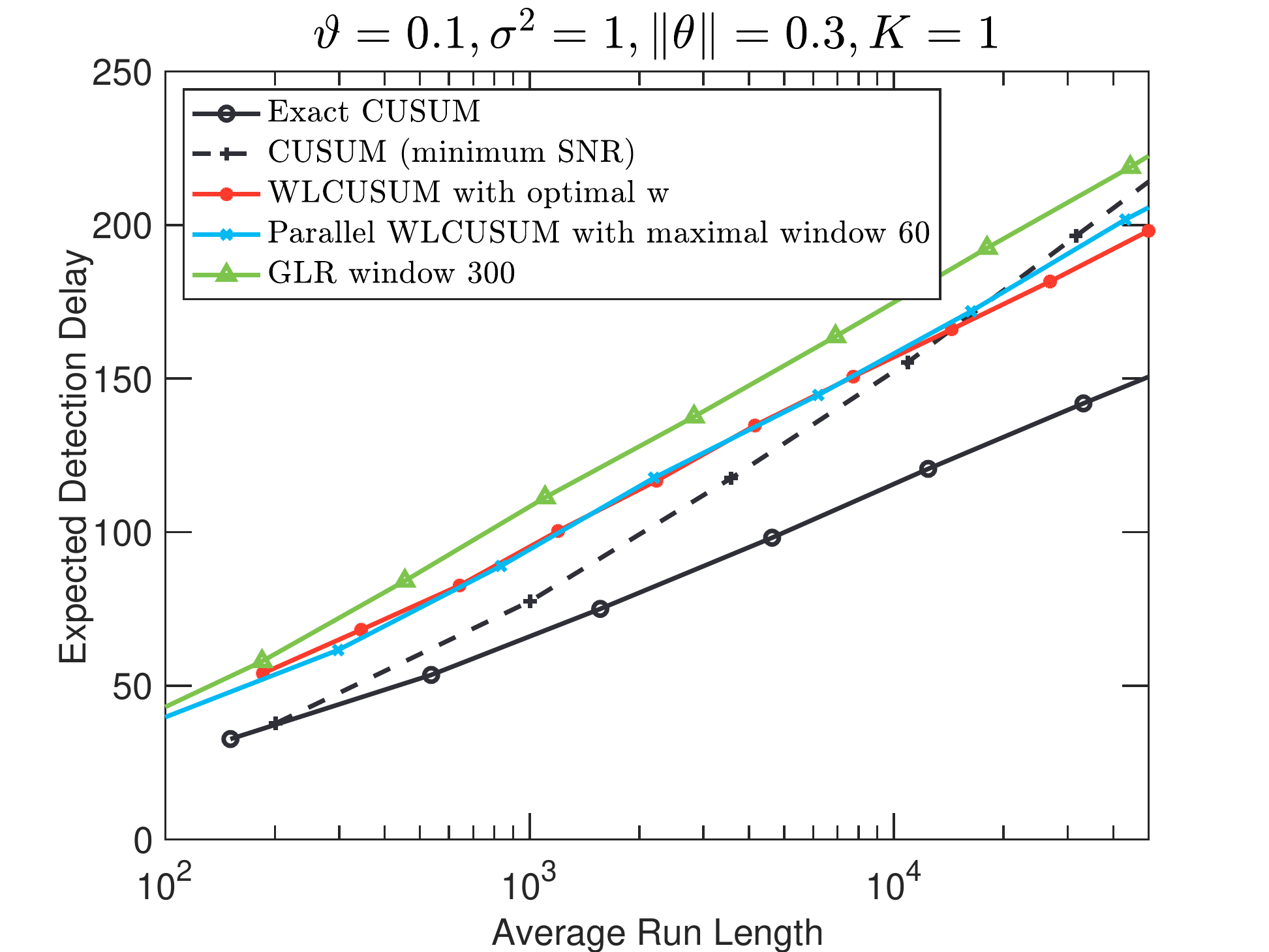}\\
(a) & (b) \\
\end{tabular}
\caption{Detection delays of different methods for univariate normal mean-shift with (a)~$\theta=0.5$ and (b)~$\theta=0.3$.}
\label{fig:fig2}
\end{figure}
In Figure\,\ref{fig:fig2}(a) and (b) we performed similar experiments but under more difficult detection scenarios.
In particular in (a) we consider the change in the mean from 0 to $\theta=0.5$ with a barrier value of $\vartheta=0.25$ while in (b) from 0 to $\theta=0.3$ with barrier $\vartheta=0.1$. Again both versions of WLCUSUM have comparable performance however now the window-limited GLR, despite its high complexity as a result of the corresponding large windows, does not enjoy better performance than WLCUSUM. In fact in the second more challenging detection example it is clearly worse than WLCUSUM. In addition, we see that the CUSUM under minimum signal strength also becomes much worse as ARL becomes large and is worse than WLCUSUM in the second more challenging detection example.

\subsection{Varying Dimensions}

To illustrate the performance on multivariate distributions, we performed simulations with varying parameter dimensions. As shown in \eqref{eq:CRLB}, the dimension of $\theta$ will affect the appropriate choice of window size as well as the quantities $\hat{\I}_0$, $\hat{\I}_\infty$, and $\hat{\!\J}_0$. We consider two dimensions: $K=5$ and $K=10$. In both cases we simulate two different mean changes (i)~from 0 to $\theta$ with $\left\Vert \theta\right\Vert = 1$ and barrier $\vartheta=0.5$, and (ii) from 0 to $\theta$ with $\left\Vert\theta\right\Vert=0.3$ and $\vartheta=0.1$. The appropriate range of window sizes is chosen based on Lemma\,\ref{lem:w_constant}.

Figure\,\ref{fig:fig4}(a) and (b) depicts the detection delay of the competing schemes under a large change namely from 0 to $\theta$ with $\|\theta\|=1$ and barrier value $\vartheta=0.5$ for the two lengths $K=5,10$. We can see in Figure\,\ref{fig:fig4}(a) and (b) that in both cases the window-limited GLR exhibits a smaller detection delay compared to WLCUSUM but, as we explained, at the expense of a significantly higher computational complexity. However this advantage tends to disappear when the detection problem becomes more challenging as in the case of the small change in the mean from 0 to $\|\theta\|=0.3$ with barrier $\vartheta=0.1$. This is clearly depicted in Figure\,\ref{fig:fig4}(c) and (d) where WLCUSUM is antagonistic to GLR for both lengths $K=5,10$. Of course we must not forget that this performance of WLCUSUM is enjoyed at a significantly lower computational complexity level as compared to the window-limited GLR.
\begin{figure*}[!t]
\centering
\begin{tabular}{cc}
\includegraphics[width = 0.45\textwidth]{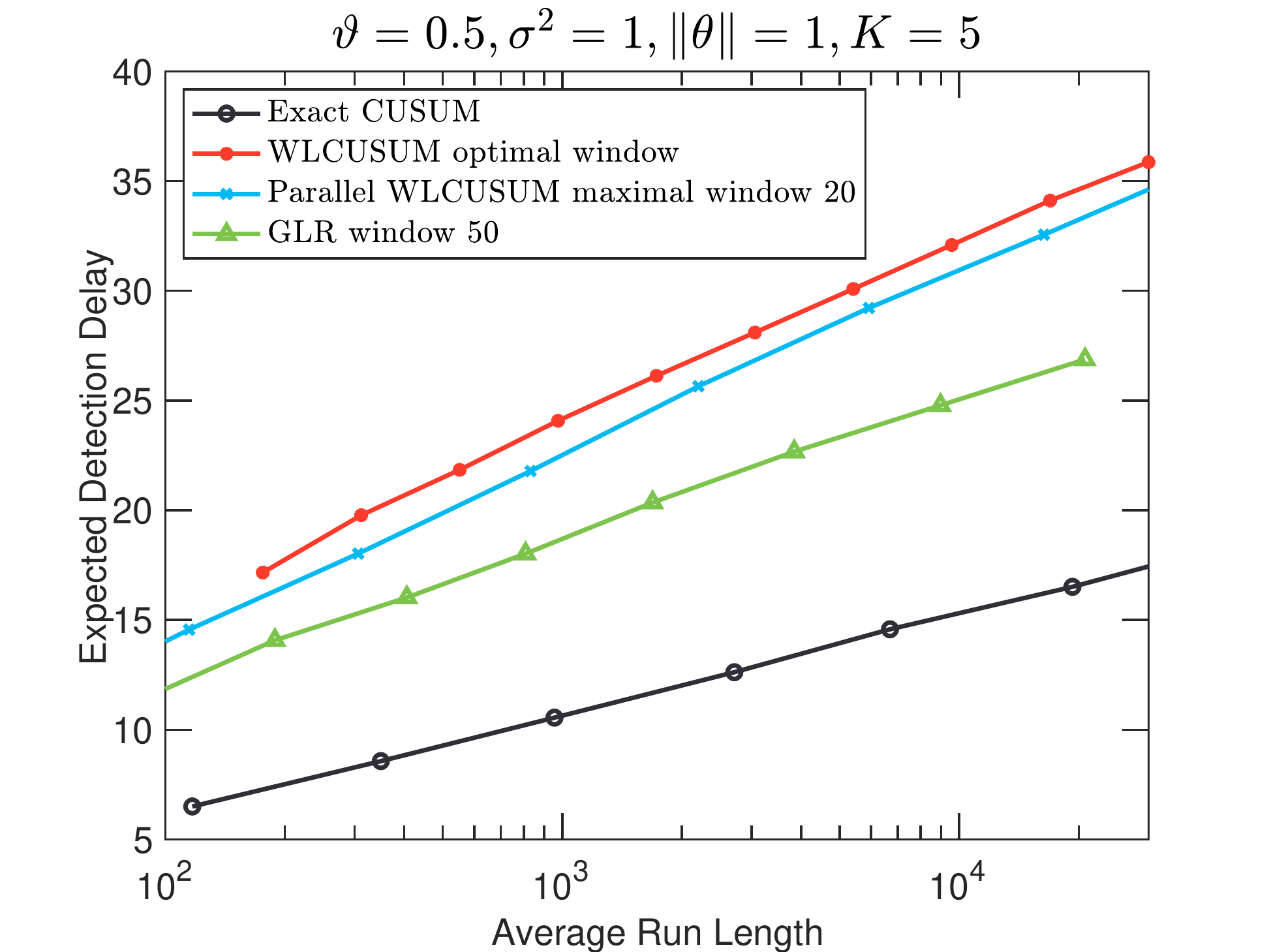} & \includegraphics[width = 0.45\textwidth]{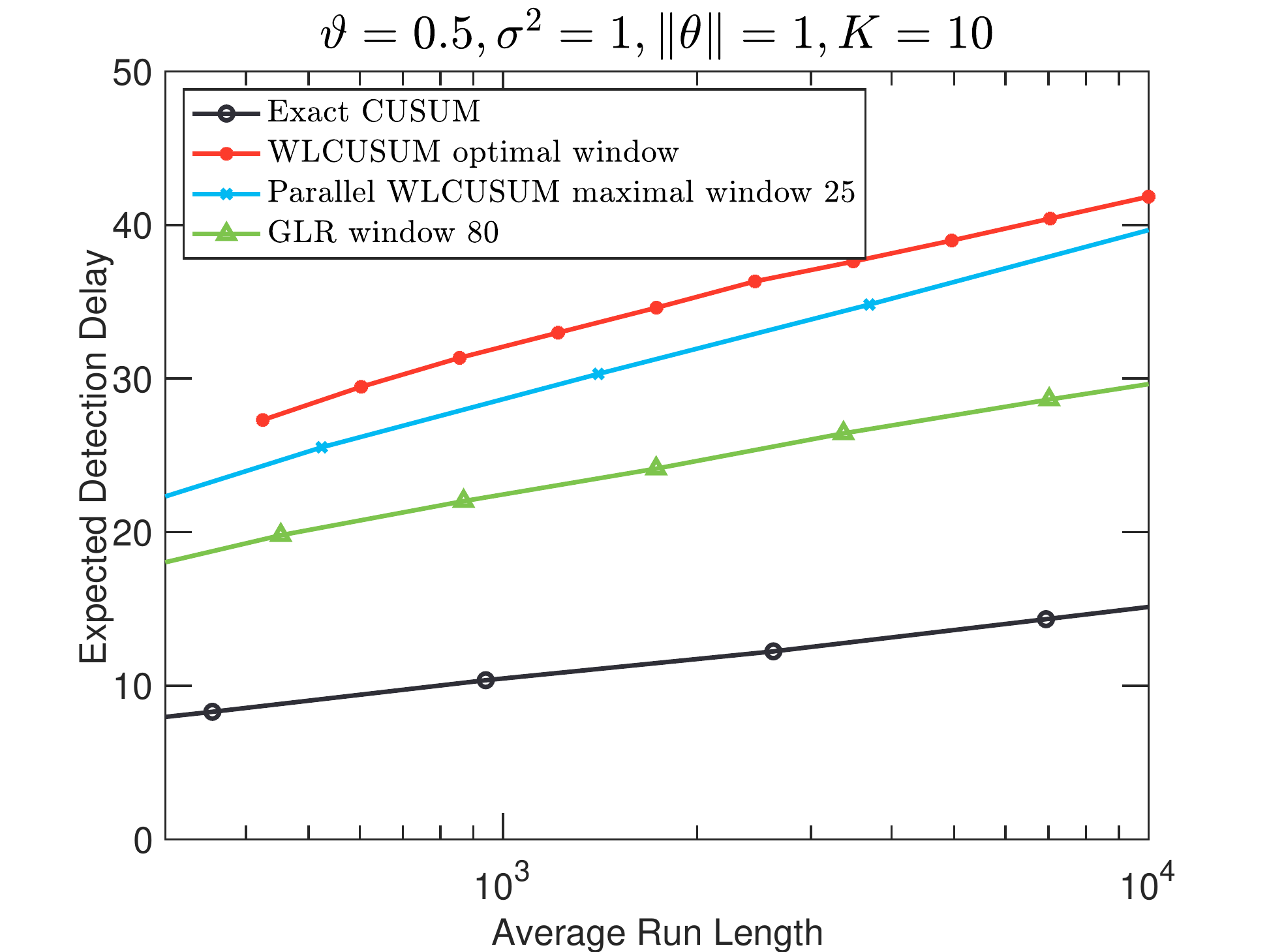} \\
(a)& (b)
\end{tabular}
\vskip0.3cm
%
%
%
\centering
\begin{tabular}{cc}
\includegraphics[width = 0.45\textwidth]{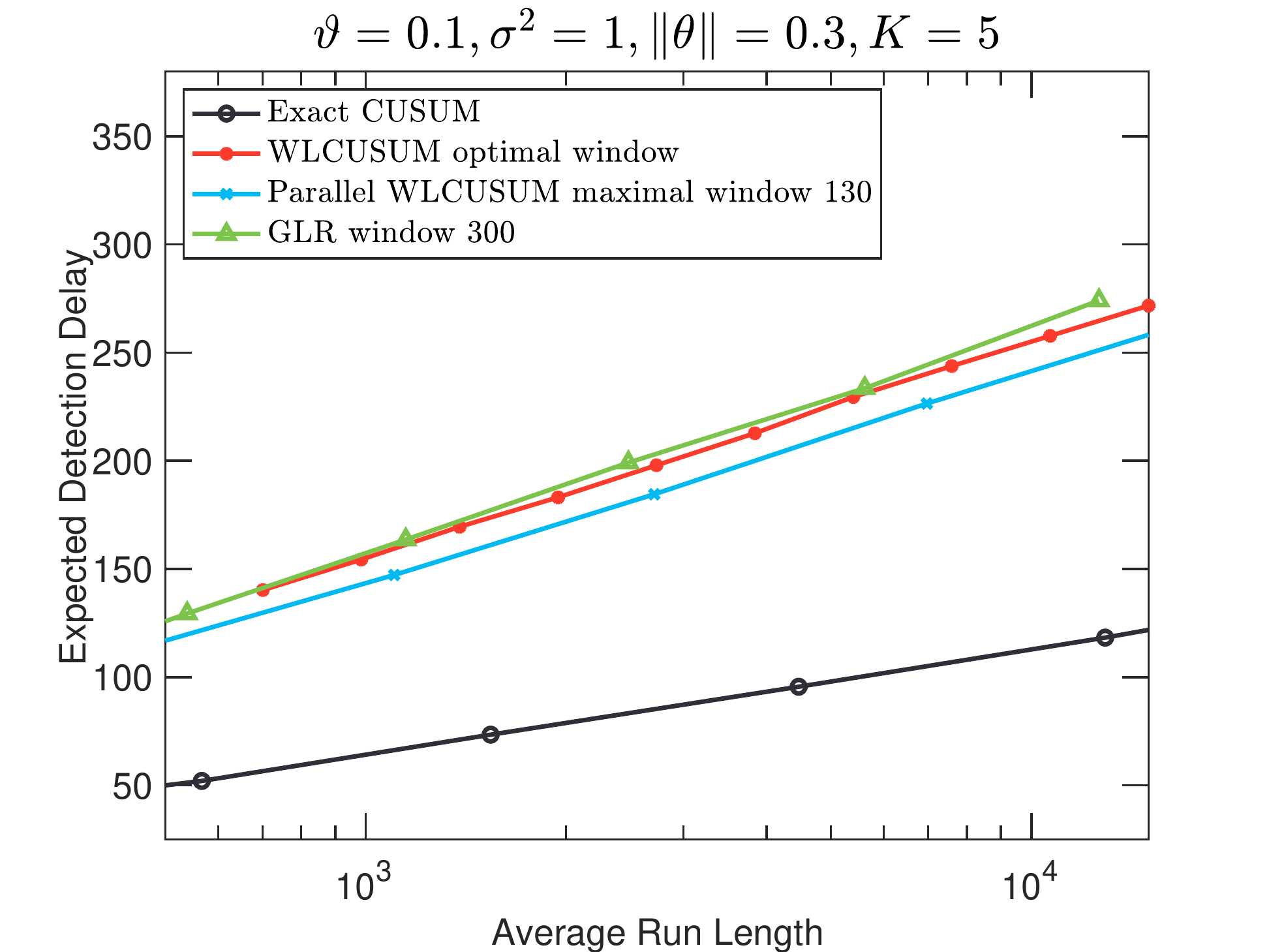} & \includegraphics[width = 0.45\textwidth]{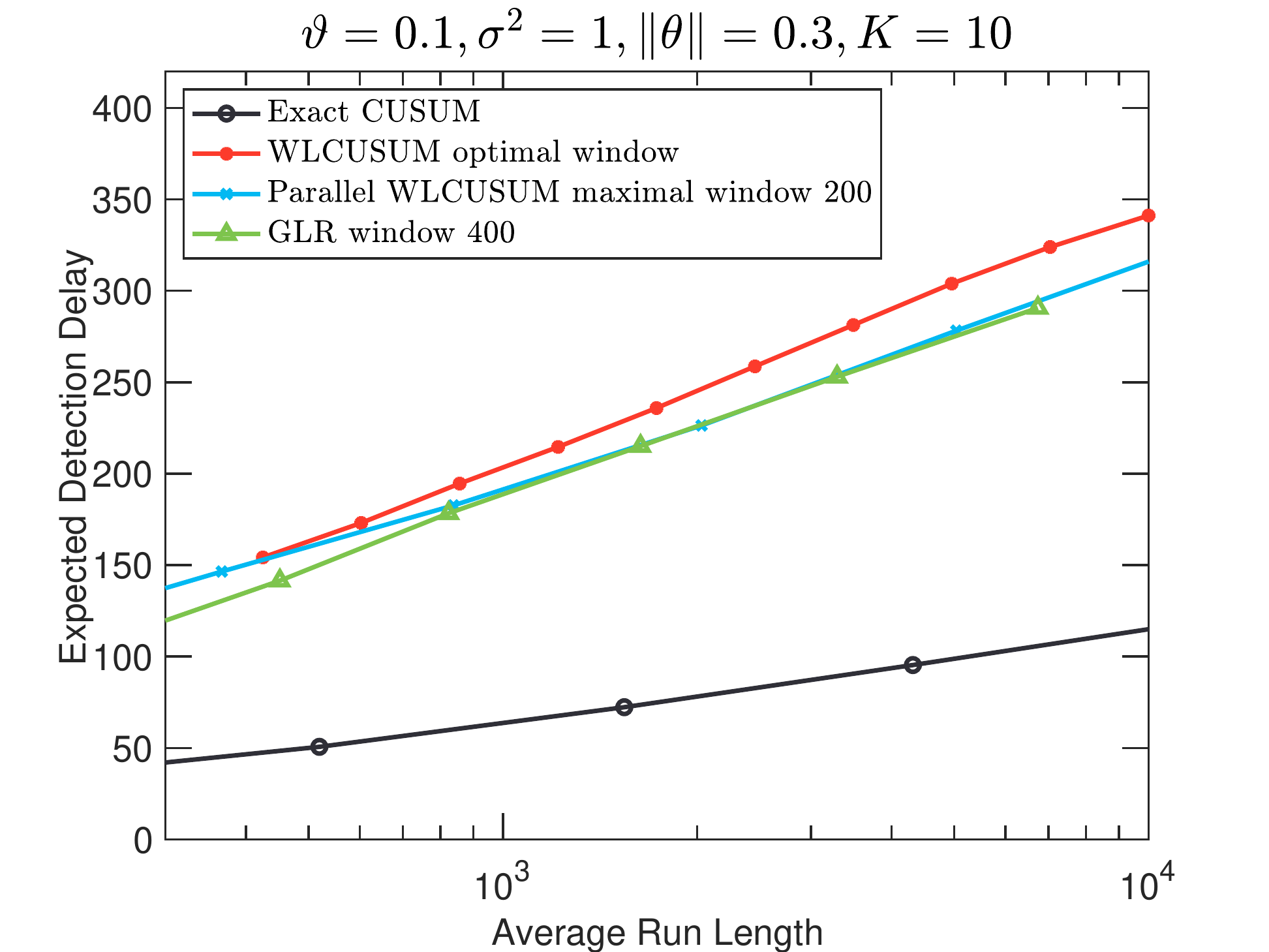} \\
(c)& (d)
\end{tabular}
\caption{Detection delay for multivariate normal mean-shift from 0 to $\theta$ with (a)~$\|\theta\|=1,~\vartheta=0.5,~K=5$; (b)~$\|\theta\|=1,~\vartheta=0.5,~K=10$;  (c)~$\|\theta\|=0.3,~\vartheta=0.1,~K=5$; (d)~$\|\theta\|=0.3,~\vartheta=0.1,~K=10$.}
%
\label{fig:fig4}
\end{figure*}

\subsection{Different Parametric Families}

We also consider the case with non-Gaussian pre-change distribution and Gaussian post-change distributions, i.e., the pre- and post-change does not belong to the same parametric family. In this case, the post-change parameter set $\Theta$ can be set as large as the entire space $\Real^K$. We assume the pre-change distribution is a Laplace distribution with mean zero and variance 1, i.e., the density function $f_\infty = \frac{1}{\sqrt{2}}\exp\{-\sqrt{2}|x|\}$; and the post-change distribution is a normal distribution $\N(\mu,\sigma^2)$. Therefore, the pre-change distribution does not belong to the post-change distribution family for any $\mu$ and $\sigma$, and we can safely choose the parameter set $\Theta$ as the entire parameter space.

In Figure\,\ref{fig:fig5}(a) and (b) we performed experiments under two different detection scenarios. In particular in (a) we consider the post-change distribution being $\mathcal N(0,1)$, i.e., although the distribution shifts from Laplace to Normal, their mean and variance remain the same. We assume the post-change variance $\sigma^2$ is known and only estimate the mean parameter during the detection procedure. Furthermore, in (b) we consider the post-change distribution being $\mathcal N(0,4)$ and assume $\sigma^2$ is unknown, so we have to estimate the post-change mean and variance jointly through maximum likelihood estimate. In such case, we have the parameter space $\Theta=\{-\infty<\mu<\infty,\sigma^2>0\}$. Again both versions of WLCUSUM have comparable performance. The window-limited GLR, despite its high complexity as a result
of the corresponding large windows (especially for (a) where the change is challenging to detect), does not enjoy better performance than WLCUSUM.

\begin{figure}[ht!]
\centering
\begin{tabular}{cc}
\includegraphics[width = 0.45\linewidth]{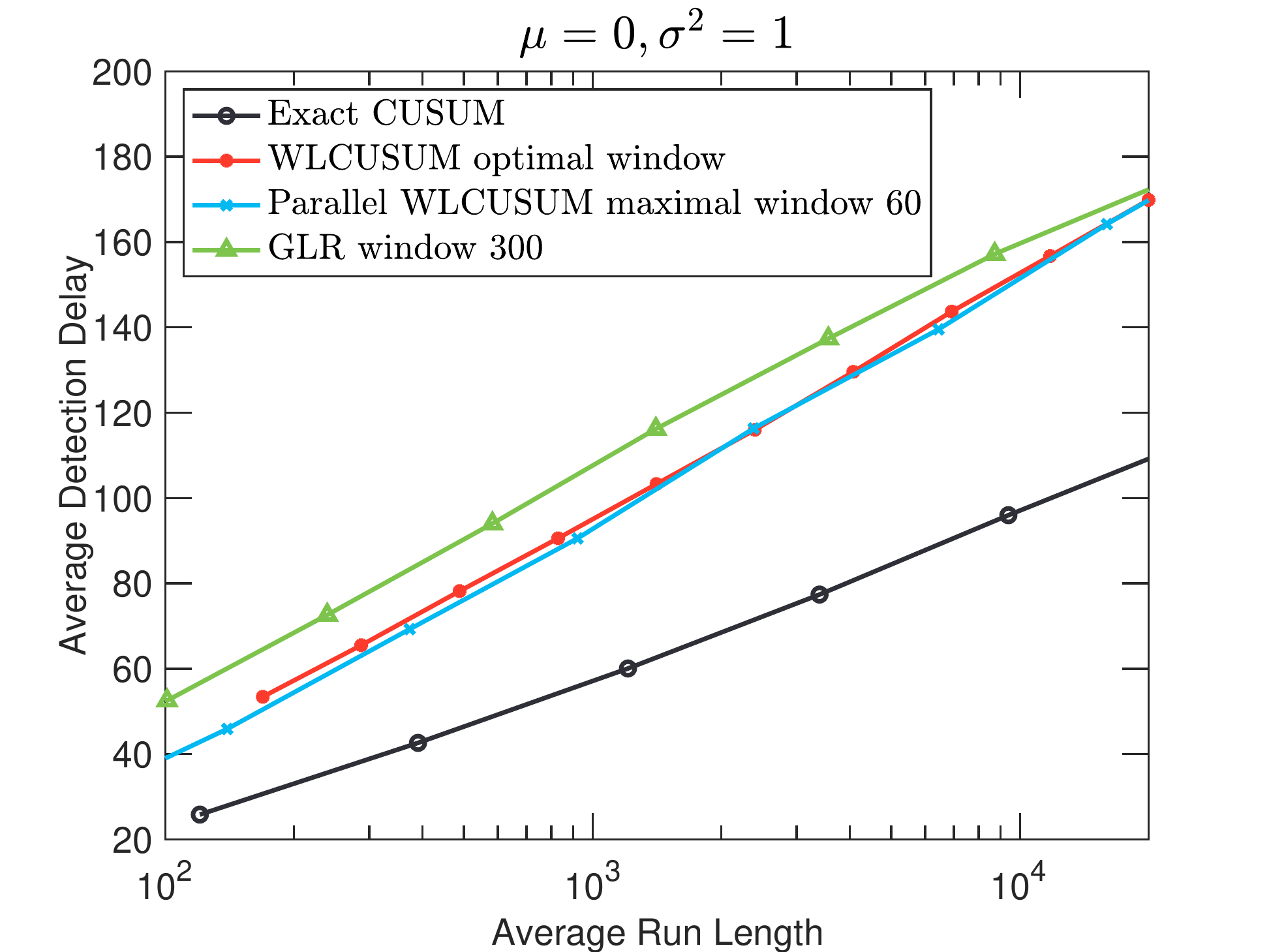} & \includegraphics[width = 0.45\linewidth]{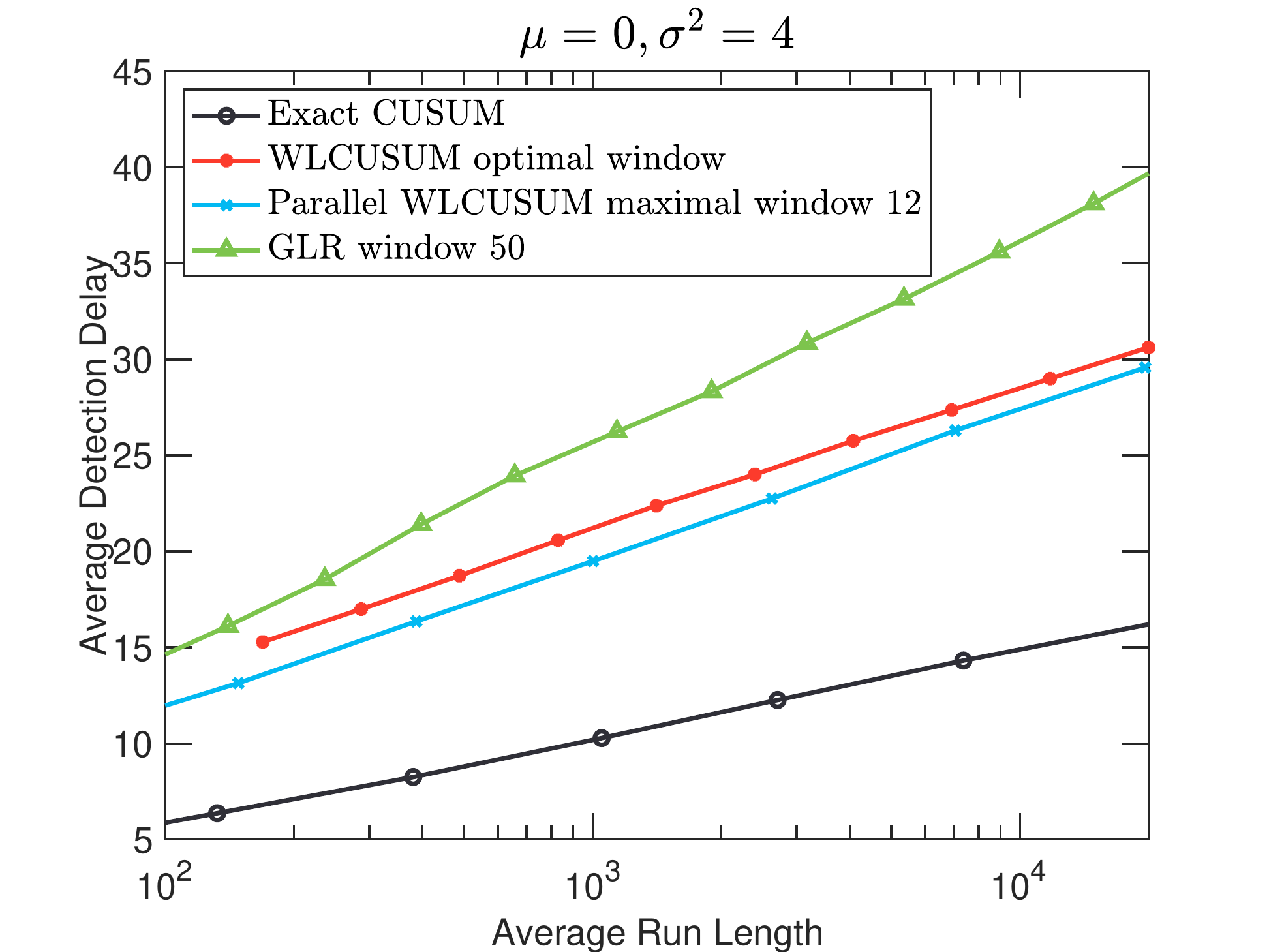}\\
(a) & (b)\\
\end{tabular}
\caption{Detection delays of different methods for univariate shift from Laplace to Normal distributions with (a) $\mu=0$, $\sigma^2=1$ and assume known $\sigma^2$; (b) $\mu=0$, $\sigma^2=4$ and assume unknown $\sigma^2$.}
\label{fig:fig5}
\end{figure}

\section{Conclusion}
In this work we consider the sequential change detection problem with known pre-change distribution and unknown post-change distribution but in certain parametric forms. We propose a window limited CUSUM procedure that uses a sliding window to perform online estimate for the unknown post-change parameter. A careful analysis on the average run length and detection delay shows the asymptotic optimality of the proposed method, with a window size much smaller than that required by window-limited GLR approach. The proposed framework also opens opportunities for several future research directions. For example, we may consider more efficient (such as ``one-sample update'' via stochastic gradient descent) estimate of the post-change parameter. Moreover, we may also try to relax the current constraint that the pre-change distribution must be a known distance away from the post-change distribution sets. The extension to joint detection and estimation is also worthwhile investigating in the future.

\appendix

\begin{proof}[Proof of Theorem\,\ref{th:1}]
Recall the definition of the process $\{\U_t\}$ is $
\U_t=\U_{t-1}+\log\frac{\f_0(\xi_t,\hth_{t-1})}{\f_\infty(\xi_t)}$ with $\U_w=0$, and the corresponding stopping time $\T'=\inf\{t>w:\U_t\geq\nu\}$. For the expectation $\Exp_0^\theta[\U_{\T'}]$ we cannot apply the usual form of Wald's identity because the terms under the sum are $w$-dependent. Indeed, while $\log\frac{\f_0(\xi_t,\theta)}{\f_\infty(\xi_t)}$ is independent from $\cF_{t-1}$, when we replace the unknown $\theta$ with the estimate $\hth_{t-1}$ then
$\log\frac{\f_0(\xi_t,\hth_{t-1})}{\f_\infty(\xi_t)}$ depends on $\cF_{t-1}$ and is independent only from $\cF_{t-w-1}$ due to the $w$-dependency of the estimates $\{\hth_t\}$. There exists version of Wald's identity for dependent samples \cite{moustakides1999extension} and for our analysis we are going to borrow ideas from this work but we intend to present all the details for our particular case.

For simplicity let us denote with $\ell_t=\log\frac{\f_0(\xi_t,\hth_{t-1})}{\f_\infty(\xi_t)}$, then we observe that for $t>w$ we can write
\begin{equation}
\begin{aligned}
\Exp_0^\theta[\U_{\T'}] =\Exp_0^\theta\left[\sum_{t=w+1}^{\T'}\ell_t\right] =\Exp_0^\theta\left[\sum_{t=w+1}^{\T'+w}\ell_t\right]
-\Exp_0^\theta\left[\sum_{t=\T'+1}^{\T'+w}\ell_t\right].    
\end{aligned}
\label{eq:AA1}
\end{equation}
Consider the two terms in \eqref{eq:AA1} separately. For the first we have
\begin{align*} &\Exp_0^\theta\left[\sum_{t=w+1}^{\T'+w}\ell_t\right]
=\Exp_0^\theta\left[\sum_{t=w+1}^{\infty}\ell_t\ind{\T'\geq t-w}
\right]=\Exp_0^\theta\left[\sum_{t=w+1}^{\infty}\Exp_0^\theta\left[\ell_t|\cF_{t-w-1}\right]\ind{\T'\geq t-w}\right]\\
&=\Exp_0^\theta\left[\sum_{t=w+1}^{\infty}\Exp_0^\theta\left[\ell_t\right]\ind{\T'\geq t-w}\right] =\Exp_0^\theta\left[\ell_{w+1}\right]\Exp_0^\theta[\T']
=\hat{\I}_0\Exp_0^\theta[\T'],
\end{align*}
with the second equality being true since $\ind{\T'\geq t-w}$ is $\cF_{t-w-1}$-measurable and the third equality being valid because $\ell_t=\log\frac{\f_0(\xi_t,\hth_{t-1})}{\f_\infty(\xi_t)}$ is independent from $\cF_{t-w-1}$.

Consider now the second term in \eqref{eq:AA1}, we observe that
\begin{equation}
\begin{aligned}
&\Exp_0^\theta \left[\sum_{t=\T'+1}^{\T'+w}\ell_t\right]=\Exp_0^\theta\!\left[\sum_{t=w+1}^{\infty}\ell_t\ind{t>\T'}\ind{\T'\geq t-w}
\right] \\
&=\Exp_0^\theta\left[\sum_{t=w+1}^{\infty}\Exp_0^\theta\left[\ell_t|\cF_{t-1}\right]\ind{t>\T'}\ind{\T'\geq t-w}
\right]\\
&\leq\Exp_0^\theta\left[\sum_{t=w+1}^{\infty}\I_0\ind{t>\T'}\ind{\T'\geq t-w}
\right]= w\I_0,    
\end{aligned}
\label{eq:mbifla1}
\end{equation}
where for the inequality we used \eqref{eq:inequ}. Combining the two expressions we conclude that
\begin{equation}
\begin{aligned}
& \hat{\I}_0\Exp_0^\theta[\T']-w\I_0\leq\Exp_0^\theta[\U_{\T'}] 
 \Rightarrow\Exp_0^\theta[\T']\leq\frac{\Exp_0^\theta[\U_{\T'}]+w\I_0}{\hat{\I}_0}=
\frac{\Exp_0^\theta[\U_{\T'}-\nu]+\nu+w\I_0}{\hat{\I}_0}.    
\end{aligned}
\label{eq:AAA1}
\end{equation}

The next step involves the control of the expectation of the overshoot $\R_\nu=\U_{\T'}-\nu$. In the case where $\U_t$ is a sum of i.i.d.~terms, we have from \cite{lorden1970excess} an elegant result that bounds the average overshoot uniformly over all $\nu$ by a constant. Unfortunately, we were not able to produce a similar conclusion for the $w$-dependent case. Instead we developed an upper bound that increases as $\sqrt{\nu}$. The good news is that even with this cruder bound, the asymptotic characteristics of our scheme will turn out to be of the same order as the ones we would have enjoyed with a constant upper bound applied on the average overshoot.

We borrow ideas from \cite{lorden1970excess} and modify them to accommodate the $w$-dependency. For any threshold $x>0$ define $\T'_x$ to be the corresponding stopping time
$$
\T'_x=\inf\{t>w:\U_t\geq x\},
$$
and denote the overshoot function as $\R_x=\U_{\T'_x}-x$. Define a sequence of stopping times $\{\tau_j\}$ with $\tau_0=w$ and 
$$
\tau_j=\inf\left\{t>\tau_{j-1}:\sum_{s=\tau_{j-1}+1}^t\ell_s>0\right\},~j\geq1,
$$
and the corresponding ladder variables $z_j=\sum_{s=\tau_{j-1}+1}^{\tau_j}\ell_s>0$. Due to the positivity of $\hat{\I}_0=\Exp_0^\theta[\ell_t]$ we have that under the $\Pro_0^\theta$ measure the stopping times are all a.s.~finite. We can now see that $\U_{\tau_j}=\sum_{i=1}^jz_i$, in fact $\U_t$ increases only at the stopping times $\{\tau_j\}$. For any given threshold $\nu$, in order to stop at $\T_\nu'$ the statistic $\U_t$ needs an increase at $\T_\nu'$, which means that there exists a random index $j_\nu$ such that $\tau_{j_\nu}=\T_\nu'$. Due to this fact we can write $\R_\nu=\sum_{j=1}^{j_\nu}z_j-\nu$. Following the same steps as in \cite{lorden1970excess}, we observe that $\R_x$ is a piecewise linear function and all pieces having slope $-1$, thus we have
\begin{equation}
\begin{aligned}
\int_0^\nu\R_x\,dx & =\int_0^{\U_{\T'_\nu}} \R_x \, dx - \int_\nu^{\U_{\T'_\nu}} \R_x \, dx =\frac{1}{2}\left\{\sum_{j=1}^{j_\nu}z_j^2-\R_\nu^2\right\}.    
\end{aligned}
\label{eq:AAA2}
\end{equation}
By definition $0<z_j$ and $z_j=\sum_{t=\tau_{j-1}+1}^{\tau_j-1}\ell_t+\ell_{\tau_j}$ with $\sum_{t=\tau_{j-1}+1}^{\tau_j-1}\ell_t\leq0$, then we have $0<\ell_{\tau_j}$ and $0<z_j\leq\ell_{\tau_j}$, which implies $z_j^2\leq\ell_{\tau_j}^2\leq\sum_{t=\tau_{j-1}+1}^{\tau_j}\ell_t^2$. Substituting in \eqref{eq:AAA2} yields
\[
2\int_0^\nu\R_x\,dx\leq\sum_{j=1}^{j_\nu}\sum_{t=\tau_{j-1}+1}^{\tau_j}\ell_t^2-\R_\nu^2 
=\sum_{t=w+1}^{\T'}\ell_t^2-\R_\nu^2 
\leq\sum_{t=w+1}^{\T'+w}\ell_t^2-\R_\nu^2.    
\]
If we take the expectation of the previous expression and use Jensen's inequality on the last term we obtain
\begin{align*}
&0\leq2\int_0^\nu\Exp_0^\theta[\R_x]\,dx
\leq\Exp_0^\theta\left[\sum_{t=w+1}^{\T'+w}\ell_t^2\right]-(\Exp_0^\theta[\R_\nu])^2
 \\
 & =\frac{\hat{\!\J}_0}{\hat{\I}_0}\Exp_0^\theta\left[\sum_{t=w+1}^{\T'+w}\ell_t\right]-(\Exp_0^\theta[\R_\nu])^2=\frac{\hat{\!\J}_0}{\hat{\I}_0}\Exp_0^\theta\left[\sum_{t=w+1}^{\T'}\ell_t+\sum_{t=\T'+1}^{\T'+w}\ell_t\right]-(\Exp_0^\theta[\R_\nu])^2 \\
& =\frac{\hat{\!\J}_0}{\hat{\I}_0}\Exp_0^\theta\left[\U_{\T'}+\sum_{t=\T'+1}^{\T'+w}\ell_t\right]-(\Exp_0^\theta[\R_\nu])^2=\frac{\hat{\!\J}_0}{\hat{\I}_0}\Exp_0^\theta\left[\R_\nu+\nu+\sum_{t=\T'+1}^{\T'+w}\ell_t\right]-(\Exp_0^\theta[\R_\nu])^2 \\
& =\frac{\hat{\!\J}_0}{\hat{\I}_0}\left\{\Exp_0^\theta[\R_\nu]+\nu+\Exp_0^\theta\left[\sum_{t=\T'+1}^{\T'+w}\ell_t\right]\right\}-(\Exp_0^\theta[\R_\nu])^2\\
& \leq\frac{\hat{\!\J}_0}{\hat{\I}_0}\left\{\Exp_0^\theta[\R_\nu]+\nu+w\I_0\right\}-(\Exp_0^\theta[\R_\nu])^2.
\end{align*}
The first equality is true because $\Exp_0^\theta[\sum_{t=w+1}^{\T'+w}\ell_t^2]=\hat{\!\J}_0\Exp_0^\theta[\T']$ and
$\Exp_0^\theta[\sum_{t=w+1}^{\T'+w}\ell_t]=\hat{\I}_0\Exp_0^\theta[\T']$. Also for the last inequality we used \eqref{eq:mbifla1}. From the nonnegativity of the integral we have
$(\Exp_0^\theta[\R_\nu])^2\leq\frac{\hat{\!\J}_0}{\hat{\I}_0}\left\{\Exp_0^\theta[\R_\nu]+\nu+w\I_0\right\}$, from which we conclude that $\Exp_0^\theta[\R_\nu]\leq\frac{\hat{\!\J}_0}{\hat{\I}_0}+\big(\frac{\hat{\!\J}_0}{\hat{\I}_0}\nu\big)^{\half}+\big(\frac{\hat{\!\J}_0}{\hat{\I}_0}\I_0w\big)^{\half}$. Given also from \eqref{eq:gamma} that $\nu=\log\gamma$, substituting in \eqref{eq:AAA1} produces the desired upper bound.
%
\end{proof}

\bibliographystyle{IEEEtran}
\bibliography{references}

\begin{thebibliography}{10}
\providecommand{\url}[1]{#1}
\csname url@samestyle\endcsname
\providecommand{\newblock}{\relax}
\providecommand{\bibinfo}[2]{#2}
\providecommand{\BIBentrySTDinterwordspacing}{\spaceskip=0pt\relax}
\providecommand{\BIBentryALTinterwordstretchfactor}{4}
\providecommand{\BIBentryALTinterwordspacing}{\spaceskip=\fontdimen2\font plus
\BIBentryALTinterwordstretchfactor\fontdimen3\font minus
  \fontdimen4\font\relax}
\providecommand{\BIBforeignlanguage}[2]{{%
\expandafter\ifx\csname l@#1\endcsname\relax
\typeout{** WARNING: IEEEtran.bst: No hyphenation pattern has been}%
\typeout{** loaded for the language `#1'. Using the pattern for}%
\typeout{** the default language instead.}%
\else
\language=\csname l@#1\endcsname
\fi
#2}}
\providecommand{\BIBdecl}{\relax}
\BIBdecl

\bibitem{poor-hadj-QCD-book-2008}
H.~V. Poor and O.~Hadjiliadis, \emph{Quickest Detection}.\hskip 1em plus 0.5em
  minus 0.4em\relax Cambridge University Press, 2008.

\bibitem{Siegmund1985}
D.~Siegmund, \emph{Sequential Analysis: Tests and Confidence Intervals}, ser.
  Springer Series in Statistics.\hskip 1em plus 0.5em minus 0.4em\relax
  Springer-Verlag, New York, 1985.

\bibitem{tartakovsky2014sequential}
A.~Tartakovsky, I.~Nikiforov, and M.~Basseville, \emph{Sequential Analysis:
  Hypothesis Testing and Changepoint Detection}.\hskip 1em plus 0.5em minus
  0.4em\relax ser. Monographs on Statistics and Applied Probability 136. Boca
  Raton, London, New York: Chapman \& Hall/CRC Press, Taylor \& Francis Group,
  2015.

\bibitem{tutorial_jsait}
L.~Xie, S.~Zou, Y.~Xie, and V.~V. Veeravalli, ``Sequential (quickest) change
  detection: Classical results and new directions,'' \emph{IEEE Journal on
  Selected Areas in Information Theory}, vol.~2, no.~2, pp. 494--514, 2021.

\bibitem{xie2019asynchronous}
L.~Xie, Y.~Xie, and G.~V. Moustakides, ``Asynchronous multi-sensor change-point
  detection for seismic tremors,'' in \emph{Proceedings of the IEEE
  International Symposium on Information Theory (ISIT)}, 2019, pp. 787--791.

\bibitem{shi2009quality}
J.~Shi and S.~Zhou, ``Quality control and improvement for multistage systems: A
  survey,'' \emph{{IIE transactions}}, vol.~41, no.~9, pp. 744--753, 2009.

\bibitem{lai1995sequential}
T.~L. Lai, ``Sequential changepoint detection in quality control and dynamical
  systems,'' \emph{J. Roy. Statist. Soc. Ser. B}, vol.~57, no.~4, pp. 613--658,
  1995.

\bibitem{balageas2010structural}
D.~Balageas, C.-P. Fritzen, and A.~G{\"u}emes, \emph{Structural health
  monitoring}.\hskip 1em plus 0.5em minus 0.4em\relax John Wiley \& Sons, 2010,
  vol.~90.

\bibitem{li2017detecting}
S.~Li, Y.~Xie, M.~Farajtabar, A.~Verma, and L.~Song, ``Detecting changes in
  dynamic events over networks,'' \emph{IEEE Trans. Signal Inform. Process.
  Netw.}, vol.~3, no.~2, pp. 346--359, 2017.

\bibitem{chandola2009anomaly}
V.~Chandola, A.~Banerjee, and V.~Kumar, ``Anomaly detection: A survey,''
  \emph{ACM computing surveys (CSUR)}, vol.~41, no.~3, pp. 1--58, 2009.

\bibitem{tartakovsky2014rapid}
A.~G. Tartakovsky, ``Rapid detection of attacks in computer networks by
  quickest changepoint detection methods,'' in \emph{Data analysis for network
  cyber-security}.\hskip 1em plus 0.5em minus 0.4em\relax Imp. Coll. Press,
  London, 2014, pp. 33--70.

\bibitem{page-biometrica-1954}
E.~S. Page, ``Continuous inspection schemes,'' \emph{Biometrika}, vol.~41, no.
  1/2, pp. 100--115, 1954.

\bibitem{mous-astat-1986}
G.~V. Moustakides, ``Optimal stopping times for detecting changes in
  distributions,'' \emph{Ann. Statist.}, vol.~14, no.~4, pp. 1379--1387, 1986.

\bibitem{lai-ieeetit-1998}
T.~L. Lai, ``Information bounds and quick detection of parameter changes in
  stochastic systems,'' \emph{IEEE Trans. Inform. Theory}, vol.~44, no.~7, pp.
  2917--2929, 1998.

\bibitem{chen2013optimal}
J.~Chen, Y.~Zhao, A.~Goldsmith, and H.~V. Poor, ``Optimal joint detection and
  estimation in linear models,'' in \emph{52nd IEEE Conference on Decision and
  Control}.\hskip 1em plus 0.5em minus 0.4em\relax IEEE, 2013, pp. 4416--4421.

\bibitem{moustakides2012joint}
G.~V. Moustakides, G.~H. Jajamovich, A.~Tajer, and X.~Wang, ``Joint detection
  and estimation: Optimum tests and applications,'' \emph{IEEE Trans. Inform.
  Theory}, vol.~58, no.~7, pp. 4215--4229, 2012.

\bibitem{Lorden1971}
G.~Lorden, ``Procedures for reacting to a change in distribution,'' \emph{Ann.
  Math. Statist.}, vol.~42, no.~6, pp. 1897--1908, 1971.

\bibitem{detectAbruptChange93}
M.~Basseville and I.~V. Nikiforov, \emph{Detection of Abrupt Changes: Theory
  and Application}, ser. Prentice Hall Information and System Sciences
  Series.\hskip 1em plus 0.5em minus 0.4em\relax Prentice Hall, Inc., Englewood
  Cliffs, NJ, 1993.

\bibitem{tartakovsky2019sequential}
A.~G. Tartakovsky, \emph{Sequential Change Detection and Hypothesis Testing:
  General Non-iid Stochastic Models and Asymptotically Optimal Rules}.\hskip
  1em plus 0.5em minus 0.4em\relax ser. Monographs on Statistics and Applied
  Probability 165. Boca Raton, London, New York: Chapman \& Hall/CRC Press,
  Taylor \& Francis Group, 2020.

\bibitem{veeravalli2013quickest}
V.~V. Veeravalli and T.~Banerjee, ``Quickest change detection,'' \emph{Academic
  Press Library in Signal Processing: {A}rray and Statistical Signal
  Processing}, vol.~3, pp. 209--255, 2014.

\bibitem{Moustakides2008}
G.~V. Moustakides, ``Sequential change detection revisited,'' \emph{Ann.
  Statist.}, vol.~36, no.~2, pp. 787--807, 2008.

\bibitem{tartakovsky2010state}
A.~G. Tartakovsky and G.~V. Moustakides, ``State-of-the-art in {B}ayesian
  changepoint detection,'' \emph{Sequential Anal.}, vol.~29, no.~2, pp.
  125--145, 2010.

\bibitem{Shiryaev1963}
A.~N. Shiryaev, ``On optimum methods in quickest detection problems,''
  \emph{Theory Probab. Appl.}, vol.~8, no.~1, pp. 22--46, 1963.

\bibitem{poll-astat-1985}
M.~Pollak, ``Optimal detection of a change in distribution,'' \emph{Ann.
  Statist.}, vol.~13, no.~1, pp. 206--227, 1985.

\bibitem{polunchenko2010optimality}
A.~S. Polunchenko and A.~G. Tartakovsky, ``On optimality of the
  {Shiryaev--Roberts} procedure for detecting a change in distribution,''
  \emph{Ann. Statist.}, vol.~38, no.~6, pp. 3445--3457, 2010.

\bibitem{kirch2015use}
C.~Kirch and J.~T. Kamgaing, ``On the use of estimating functions in monitoring
  time series for change points,'' \emph{Journal of Statistical Planning and
  Inference}, vol. 161, pp. 25--49, 2015.

\bibitem{lucas1982combined}
J.~M. Lucas, ``Combined {Shewhart-CUSUM} quality control schemes,'' \emph{J.
  Qual. Technol.}, vol.~14, no.~2, pp. 51--59, 1982.

\bibitem{westgard1977combined}
J.~Westgard, T.~Groth, T.~Aronsson, and C.~De~Verdier, ``Combined
  {Shewhart-CUSUM} control chart for improved quality control in clinical
  chemistry.'' \emph{Clinical chemistry}, vol.~23, no.~10, pp. 1881--1887,
  1977.

\bibitem{sparks2000cusum}
R.~S. Sparks, ``{CUSUM} charts for signalling varying location shifts,''
  \emph{J. Qual. Technol.}, vol.~32, no.~2, pp. 157--171, 2000.

\bibitem{zhao2005dual}
Y.~Zhao, F.~Tsung, and Z.~Wang, ``Dual {CUSUM} control schemes for detecting a
  range of mean shifts,'' \emph{{IIE transactions}}, vol.~37, no.~11, pp.
  1047--1057, 2005.

\bibitem{yu2020note}
Y.~Yu, O.~H.~M. Padilla, D.~Wang, and A.~Rinaldo, ``A note on online change
  point detection,'' \emph{arXiv preprint arXiv:2006.03283}, 2020.

\bibitem{romano2021fast}
G.~Romano, I.~Eckley, P.~Fearnhead, and G.~Rigaill, ``Fast online changepoint
  detection via functional pruning {CUSUM} statistics,'' \emph{arXiv preprint
  arXiv:2110.08205}, 2021.

\bibitem{nikiforov1995generalized}
I.~V. Nikiforov, ``A generalized change detection problem,'' \emph{IEEE Trans.
  Inform. Theory}, vol.~41, no.~1, pp. 171--187, 1995.

\bibitem{lai2000sequential}
T.~L. Lai, ``Sequential multiple hypothesis testing and efficient fault
  detection-isolation in stochastic systems,'' \emph{IEEE Transactions on
  Information Theory}, vol.~46, no.~2, pp. 595--608, 2000.

\bibitem{abbasi2019optimal}
S.~Abbasi and A.~Haq, ``Optimal {CUSUM} and adaptive {CUSUM} charts with
  auxiliary information for process mean,'' \emph{J. Stat. Comput. Simul.},
  vol.~89, no.~2, pp. 337--361, 2019.

\bibitem{abbasi2020new}
------, ``New adaptive {CUSUM} charts for process mean,'' \emph{Comm. Statist.
  Simulation Comput.}, vol.~49, no.~11, pp. 2944--2962, 2020.

\bibitem{jiang2008adaptive}
W.~Jiang, L.~Shu, and D.~W. Apley, ``Adaptive {CUSUM} procedures with
  {EWMA}-based shift estimators,'' \emph{{IIE Transactions}}, vol.~40, no.~10,
  pp. 992--1003, 2008.

\bibitem{luo2009adaptive}
Y.~Luo, Z.~Li, and Z.~Wang, ``Adaptive {CUSUM} control chart with variable
  sampling intervals,'' \emph{Comput. Statist. Data Anal.}, vol.~53, no.~7, pp.
  2693--2701, 2009.

\bibitem{shu2006markov}
L.~Shu and W.~Jiang, ``A {Markov} chain model for the adaptive {CUSUM} control
  chart,'' \emph{J. Qual. Technol.}, vol.~38, no.~2, pp. 135--147, 2006.

\bibitem{wu2009enhanced}
Z.~Wu, J.~Jiao, M.~Yang, Y.~Liu, and Z.~Wang, ``An enhanced adaptive {CUSUM}
  control chart,'' \emph{{IIE transactions}}, vol.~41, no.~7, pp. 642--653,
  2009.

\bibitem{cao2018entropy}
Y.~Cao, L.~Xie, Y.~Xie, and H.~Xu, ``Sequential change-point detection via
  online convex optimization,'' \emph{Entropy}, vol.~20, no.~2, 2018.

\bibitem{lorden2008sequential}
G.~Lorden and M.~Pollak, ``Sequential change-point detection procedures that
  are nearly optimal and computationally simple,'' \emph{Sequential Anal.},
  vol.~27, no.~4, pp. 476--512, 2008.

\bibitem{wu2017detecting}
Y.~Wu, ``Detecting changes in a multiparameter exponential family by using
  adaptive {CUSUM} procedure,'' \emph{Sequential Anal.}, vol.~36, no.~4, pp.
  467--480, 2017.

\bibitem{xu2021multi}
Q.~Xu and Y.~Mei, ``Multi-stream quickest detection with unknown post-change
  parameters under sampling control,'' in \emph{2021 IEEE International
  Symposium on Information Theory (ISIT)}.\hskip 1em plus 0.5em minus
  0.4em\relax IEEE, 2021, pp. 112--117.

\bibitem{silverman1986density}
B.~W. Silverman, \emph{Density estimation for statistics and data analysis},
  ser. Monographs on Statistics and Applied Probability.\hskip 1em plus 0.5em
  minus 0.4em\relax Chapman \& Hall, London, 1986.

\bibitem{xu2021optimum}
Q.~Xu, Y.~Mei, and G.~V. Moustakides, ``Optimum multi-stream sequential
  change-point detection with sampling control,'' \emph{IEEE Trans. Inform.
  Theory}, vol.~67, no.~11, pp. 7627--7636, 2021.

\bibitem{janson1984runs}
S.~Janson, ``Runs in $ m $-dependent sequences,'' \emph{The Annals of
  Probability}, vol.~12, no.~3, pp. 805--818, 1984.

\bibitem{xie2020seq_ana}
L.~Xie, Y.~Xie, and G.~V. Moustakides, ``Sequential subspace change point
  detection,'' \emph{Sequential Anal.}, vol.~39, no.~3, pp. 307--335, 2020.

\bibitem{jacob2008sequential}
T.~Jacob and R.~K. Bansal, ``Sequential change detection based on universal
  compression algorithms,'' in \emph{2008 IEEE International Symposium on
  Information Theory}.\hskip 1em plus 0.5em minus 0.4em\relax IEEE, 2008, pp.
  1108--1112.

\bibitem{lehmann2006theory}
E.~L. Lehmann and G.~Casella, \emph{Theory of point estimation}.\hskip 1em plus
  0.5em minus 0.4em\relax Springer Science \& Business Media, 2006.

\bibitem{lorden1970excess}
G.~Lorden, ``On excess over the boundary,'' \emph{Ann. Math. Statist.},
  vol.~41, no.~2, pp. 520--527, 1970.

\bibitem{moustakides1999extension}
G.~V. Moustakides, ``Extension of {Wald}'s first lemma to {Markov} processes,''
  \emph{J. Appl. Probab.}, vol.~36, no.~1, pp. 48--59, 1999.

\end{thebibliography}

\end{document}